\documentclass{article}
\usepackage{amsfonts}
\usepackage{amsmath,amsthm}

\usepackage[all]{xy}
\usepackage{graphicx}

\usepackage{amssymb}
\usepackage{latexsym}

\DeclareMathAlphabet\EuFrak{U}{euf}{m}{n}	%  the Bold Euler Fraktur
\SetMathAlphabet\EuFrak{bold}{U}{euf}{b}{n}	%  gothic font

\newcommand{\hra}{\hookrightarrow}

\newcommand{\ovl}{\overline}

\newcommand{\wa}{\widehat}

\newcommand{\sC}{{\it C*}-}
\newcommand{\bC}{{\mathbb C}}

\newcommand{\bZ}{{\mathbb Z}}

\newcommand{\bN}{{\mathbb N}}
\newcommand{\bP}{{\mathbb P}}
\newcommand{\bS}{{\mathbb S}}
\newcommand{\bSO}{{\mathbb{SO}}}

\newcommand{\ud}{{{\mathbb U}(d)}}
\newcommand{\sud}{{{\mathbb {SU}}(d)}}

\newcommand{\eps}{\varepsilon}

\newcommand{\mA}{\mathcal A}
\newcommand{\mB}{\mathcal B}
\newcommand{\mC}{\mathcal C}
\newcommand{\mE}{E}
\newcommand{\mF}{\mathcal F}
\newcommand{\mH}{\mathcal H}

\newcommand{\mL}{\mathcal L}
\newcommand{\mM}{\mathcal M}
\newcommand{\mN}{\mathcal N}
\newcommand{\mO}{\mathcal O}

\newcommand{\mR}{\mathcal R}
\newcommand{\mS}{\mathcal S}

\newcommand{\mZ}{\mathcal Z}

\newcommand{\mG}{\mathcal G}
\newcommand{\mcE}{\mathcal E}

\newcommand{\zro}{C(X^\rho)}

\newcommand{\oro}{\mO_\rho}

\newcommand{\soro}{{\mO_{\rho,\eps}}}

\newcommand{\ii}{\iota,\iota}

\newcommand{\mrs}{\mM^r,\mM^s}
\newcommand{\ers}{\mE^r,\mE^s}
\newcommand{\rhors}{\rho^r , \rho^s}

\newcommand{\taurs}{\tau_\mG^r , \tau_\mG^s}
\newcommand{\smrs}{\tau_\mM^r , \tau_\mM^s}

\newcommand{\sgrs}{\sigma_\mG^r,\sigma_\mG^s}
\newcommand{\srs}{\sigma^r , \sigma^s}
\newcommand{\sers}{\sigma_\mE^r , \sigma_\mE^s}

\newcommand{\wE}{\wa {\mE}}

\newcommand{\mcSUE}{{\mathcal {SU} \mE}}
\newcommand{\mcUE}{\mathcal {U} \mE}
\newcommand{\mSUE}{ {{\bf{SU}} \mE} }
\newcommand{\mSG}{ \mathcal{SG} }

\newcommand{\mc}{\mA \rtimes_\mu \wa \mG }

\newcommand{\coe}{\mO_{\mE}}
\newcommand{\com}{\mO_{\mM}}
\newcommand{\cog}{\mO_{\mG}}
\newcommand{\cmg}{\mO_\mM^\mG}

\newtheorem{thm}{Theorem}[section]

\newtheorem{lem}[thm]{Lemma}
\newtheorem{prop}[thm]{Proposition}

\newtheorem{defn}[thm]{Definition}

\theoremstyle{definition}
\newtheorem{ex}{Example}[section]

\theoremstyle{remark}
\newtheorem{rem}{Remark}[section]

\numberwithin{equation}{section}

\begin{document}

\author{{\sf Ezio Vasselli}
                         \\{\it Dipartimento di Matematica}
                         \\{\it University of Rome "La Sapienza"}
			 \\{\it P.le Aldo Moro, 2 - 00185 Roma - Italy }
                         \\{\sf vasselli@mat.uniroma2.it}}

\title{ Gauge-equivariant Hilbert bimodules\\and\\crossed products by endomorphisms}
\maketitle

\begin{abstract}
\sC endomorphisms arising from superselection structures with non-trivial centre define a 'rank' and a 'first Chern class'. Crossed products by such endomorphisms involve the Cuntz-Pimsner algebra of a vector bundle having the above-mentioned rank and first Chern class, and can be used to construct a duality for abstract (nonsymmetric) tensor categories vs. group bundles acting on (nonsymmetric) Hilbert bimodules. Existence and unicity of the dual object (i.e., the 'gauge' group bundle) are not ensured: we give a description of this phenomenon in terms of a certain moduli space associated with the given endomorphism. The above-mentioned Hilbert bimodules are noncommutative analogues of gauge-equivariant vector bundles in the sense of Nistor-Troitsky.

\

\noindent {\em AMS Subj. Class.:} 46L05, 46L08, 22D35.


\end{abstract}

%\tableofcontents
%\markboth{Contents}{Contents}

\section{Introduction.}
\label{intro}

One of the novelties introduced in the algebraic approach to quantum field theory, in particular in the case of localized quantum charges, is the way in which superselection sectors are regarded. In fact, whilst they are usually defined as representations of the observable algebra, in the algebraic approach they become {\em endomorphisms} of the observable algebra. The main advantage of this point of view is that the phenomenon of composition of charges is conveniently described in terms of composition of endomorphisms.
Of course not all the endomorphisms of the observable algebra, that we denote by $\mA$, have a physical interpretation. One of the features characterizing the ones of physical interest is the {\em symmetry}, which describes the statistical properties of the sector in terms of a unitary representation $\eps$ of the permutation group in $\mA$ (see \cite{DR87,DR89A,DR90} and related references).

Now, if $\rho \in {\bf end} \mA$ is an endomorphism associated with a sector, and  
\begin{equation}
\label{def_int}
( \rhors ) :=
\left\{
t \in \mA : t \rho^r (a) = \rho^s (a) t 
\ , \
a \in \mA
\right\}
\ \ , \ \
r,s \in \bN
\ ,
\end{equation}
are the intertwiner spaces, then the Doplicher-Roberts theory allows one to construct a crossed product of $\mA$ by $\rho$, in terms of a \sC algebra $\mF$ generated by $\mA$ and orthogonal partial isometries $\psi_1 , \ldots ,$ $\psi_d$, $d \in \bN$, such that 
\begin{equation}
\label{eq_inn}
\rho (a) = \sum_i \psi_i a \psi_i^*
\ \ , \ \
a \in \mA
\end{equation}
(see \cite{DR89,DR89A,DR90}). A remarkable property is that $\mF$ comes equipped with a compact group action $G \to {\bf aut} \mF$ with fixed-point algebra $\mA$, such that each $( \rhors )$, $r,s \in \bN$, is interpreted as the space of $G$-invariant operators between tensor powers of a Hilbert space of dimension $d$.
In physical terms, $\mF$ plays the role of the field algebra, and $G$ is the gauge group describing the superselection structure of sectors $\rho^r$, $r \in \bN$. The pair $( \mF,G)$ is uniquely determined by the \sC dynamical system $(\mA,\rho)$.

At the mathematical level, one of the crucial properties required for the construction of $\mF$ is triviality of the centre of $\mA$. In recent times, the more general situation in which $\mA$ has a nontrivial centre $\mZ$ has been considered, sometimes coupled with a braided symmetry (in the context of low-dimensional quantum field theory, \cite{MS90}), other times in the presence of a weak form of permutation symmetry, that we call permutation {\em quasi-symmetry} (see \cite{BL97,BL04,Vas05}, or \S \ref{key_ps} of the present paper): roughly speaking, with this we mean that not all the elements of $(\rhors)$, $r,s \in \bN$, necessarily fulfill permutation symmetry.

In the present paper we present a theory for endomorphisms with permutation quasi-symmetry, following the research line of \cite{Vas04,Vas05,Vas02g}. Differently from former works, we do not make any assumption on the structure of the spaces $(\rhors)$ (these were supposed to be free $\mZ$-bimodules in \cite{BL97,BL04}, and locally trivial fields of Banach spaces in \cite{Vas06p}).

Our main construction yields a crossed product $\mF$ generated by $\mA$ and a Hilbert $\mZ$-bimodule $\mM$, whose elements play a role analogous to (\ref{eq_inn}). The bimodule $\mM$ is generally not free -- and this implies that $\rho$ does not fulfill the {\em special conjugate property} in the sense of \cite{DR89A} -- , but also non-symmetric, in the sense that the left and right $\mZ$-actions may not coincide; indeed, these coincide if and only if $\rho$ has symmetry in the usual sense. Moreover, the relation
\begin{equation}
\label{eq_relcom}
\mA' \cap \mF = \mZ
\end{equation}
is fulfilled, in accord with the principle stated in \cite{MS90} in the setting of low dimensional quantum field theory. In this case, we say that $\mF$ is a {\em Hilbert extension} of $( \mA , \rho )$.

Some facts have to be remarked. First, if we denote the fixed-point algebra of $\mZ$ with respect to the $\rho$-action by $\zro$, then we find that $\zro$ is contained in the centre of $\mA$ and $\mF$. This implies that $\mA$ and $\mF$ have a natural structure of $\zro$-algebra in the sense of Kasparov (\cite{Kas88}), with the consequence that we may regard them as bundles of \sC algebras over $X^\rho$.
Secondly, instead of a compact group we obtain a group bundle $\mG \to X^\rho$ playing the role of the gauge group; the action of $\mG$ on $\mF$ is defined in terms of to the notion of {\em fibred action} introduced in \cite{Vas02g} (roughly speaking, each fibre of $\mG$ acts on the corresponding fibre of $\mF$; in the present paper we will use the expression {\em gauge action} instead of the one of {\em fibred action}).
Finally, a crucial fact is that existence and unicity of $( \mF , \mG )$ are not ensured; we give a complete description of this phenomenon in terms of the space of sections of a certain bundle of homogeneous spaces associated with $\rho$. 

\

The present paper is organized as follows.

In \S \ref{key_group_duals} we study some properties of gauge actions on the Cuntz-Pimsner algebra of a vector bundle; this will yield our model for symmetric endomorphisms.

In \S \ref{sec_fa_hb} we introduce the notion of {\em gauge-equivariant Hilbert bimodule}. In our approach to the construction of $\mF$, gauge-equivariant Hilbert bimodules shall substitute the Hilbert spaces generated by the above-mentioned partial isometries $\psi_i$, $i = 1 , \ldots , d$. Moreover, we introduce the notion of gauge-equivariant Kasparov module, generalizing the ones of gauge-equivariant vector bundle and gauge-invariant Fredholm operator in the sense of \cite{NT04}.

In \S \ref{sec_dual_action} we analyze the way in which gauge actions interact with {\em dual actions}. If $\mE \to X$ is a vector bundle and $\mG \to X$ is a bundle of unitary automorphisms of $\mE$, then a dual action on $(\mA,\rho)$ is a functor $\mu : \wa \mG \to \wa \rho$, where $\wa \mG$ is the category with objects the tensor powers $\mE^r$, $r \in \bN$, and arrows the spaces $(\ers)_\mG$ of $\mG$-invariant morphisms from $\mE^r$ to $\mE^s$, and $\wa \rho$ is the category with objects $\rho^r$, $r \in \bN$, and arrows $(\rhors)$. Applying a variant of the construction in \cite[\S 3]{Vas05}, we construct a crossed product \sC algebra $\mc$ equipped with a gauge $\mG$-action and a gauge-equivariant Hilbert $\mZ$-bimodule $\mM \subset \mc$. This bimodule is generally non-symmetric, and may be regarded as a tensor product of the type $\wE \otimes_{C(X)} \mZ$, where $\wE$ is the module of sections of $\mE$. The pair $( \mA \rtimes_\mu \wa \mG  , \mG )$ is our model for Hilbert extensions of $(\mA,\rho)$.

In \S \ref{cp_spec} we prove our main results. The starting point is the fact that if $\rho$ is quasi-symmetric and fulfills a twisted version of the special conjugate property (\S \ref{sec_spe}), then every vector bundle $\mE \to X^\rho$ with suitable rank and first Chern class induces a dual action $\mu : \wa \mcSUE \to \wa \rho$, where $\mcSUE$ is the bundle of special unitaries of $\mE$. The spectrum of the centre of $\mA \rtimes_\mu \wa \mcSUE$ defines by Gel'fand duality a bundle $\Omega_E \to X^\rho$. In Theorem \ref{thm_dual}, we show that the space of sections of the type $s : X^\rho \hra \Omega_E$ is in one-to-one correspondence with the set of Hilbert extensions $( \mA \rtimes_\nu \wa \mG , \mG )$, $\mG \subseteq \mcSUE$. In Theorem \ref{thm_uni}, we consider Hilbert extensions $( \mA \rtimes_\nu \wa \mG , \mG )$, $( \mA \rtimes_{\nu'} \wa \mG' , \mG' )$ of $(\mA,\rho)$ and give a necessary and sufficient condition to get an isomorphism $\mG \simeq \mG'$. The third Theorem \ref{thm_moduli} yields a complete classification of Hilbert extensions of $(\mA,\rho)$ at varying of $\mE$. Finally, in Theorem \ref{car_rho} a duality is proved, characterizing each $( \rhors )$ as the space $(\mrs)_\mG$ of $\mG$-invariant operators between tensor powers of $\mM$; in particular, we show that the spaces $(\rhors)_\eps$ of intertwiners that fulfill permutation symmetry are isomorphic to the spaces $(\ers)_\mG$. Examples of non-existence and non-unicity of the Hilbert extension are given in \S \ref{ex_non_uni} and \S \ref{ex_non_exi}.

\section{Keywords and Notation.}
\label{preli}

For every set $S$, we denote the corresponding identity map by $id_S$.

\

About \sC categories and (semi)tensor \sC categories, we refer the reader to \cite{DR89,DPZ97}.

\

If $X$ is a compact Hausdorff space, then we denote the \sC algebra of continuous functions from $X$ to $\bC$ by $C(X)$. For each $x \in X$, we denote the closed ideal of functions vanishing on $x$ by $C_x(X)$. For every open $U \subset X$, we denote the ideal of functions vanishing on $X - U$ by $C_0(U) \subset C(X)$. If $\left\{ X_i  \right\}$ is an open cover of $X$, then we write $X_{ij} :=$ $X_i \cap X_j$.

\

A {\em bundle} is given by a surjective map of locally compact Hausdorff spaces $p : Y \to X$. The fibred product with a bundle $p' : Y' \to X$ is defined as the space 
\[
Y \times_X Y' \ := \ \{ (y,y') \in Y \times Y' : p(y) = p'(y')  \}
\ ,
\]
which becomes a bundle when endowed with the natural projection on $X$. The {\em fibre} of $Y$ over $x \in X$ is given by $Y_x :=$ $p^{-1}(x)$. A {\em section} of $Y$ is given by a continuous map $s : X \to Y$ such that $p \circ s = id_X$. The set of sections of $Y$ is denoted by $S_X(Y)$, and is endowed with the topology such that each map of the type $c_* : S_X(Y) \to C_0(X)$, $c_*(s) :=$ $c \circ s$, $c \in C_0(Y)$, is continuous.
For basic properties of {\em vector bundles}, we refer the reader to \cite{Ati,Kar}. In the present paper we assume that every vector bundle is endowed with a Hermitian structure. In particular, we will make use of the notion of {\em equivariant vector bundle} (\cite[\S 1.6]{Ati}), and {\em gauge-equivariant vector bundle} (\cite[\S 3]{NT04}).

\

Let $\rho \in {\bf end} \mA$, $\rho' \in {\bf end} \mA'$ be \sC endomorphisms. A \sC morphism $\eta : \mA \to \mA'$ such that $\eta \circ \rho = \rho' \circ \eta$ is denoted by $\eta : ( \mA , \rho ) \to ( \mA' , \rho' )$. We denote the identity automorphism by $\iota \in {\bf end} \mA$, and use the convention $\rho^0 := \iota$, $0 \in \bN$.

\

Let $X$ be a compact Hausdorff space. A unital $C(X)$-{\em algebra} $\mA$ is a unital \sC algebra endowed with a unital morphism $C(X) \to \mA' \cap \mA$, called the $C(X)$-{\em action} (see \cite{Kas88}). We assume that the $C(X)$-action is also injective, thus elements of $C(X)$ are regarded as elements of $\mA$. For every $x \in X$, the quotient $\pi_x : \mA \to$ $\mA_x :=$ $\mA / (C_x(X) \mA)$ is called the {\em fibre epimorphism}. 
For every open $U \subset X$, we define the {\em restriction} $\mA_U$ as the closed span of elements of the type $fa$, $f \in C_0(U)$, $a \in \mA$. Note that $\mA_U$ is a closed ideal of $\mA$.
\sC morphisms between $C(X)$-algebras equivariant with respect to the $C(X)$-actions are called $C(X)${\em -morphisms}. The set of $C(X)$-automorphisms (resp. $C(X)$-endomorphisms) of $\mA$ is denoted by ${\bf aut}_X \mA$ (resp. ${\bf end}_X \mA$).

The category of $C(X)$-algebras is equivalent to the one with objects certain topological objects called \sC {\em bundles}. A \sC bundle with base space $X$ is given by a surjective continuous map 
\[
Q : \Sigma \to X \ ,
\]
where $\Sigma$ is a Hausdorff space such that: (1) Every $\Sigma_x :=$ $Q^{-1}(x)$, $x \in X$, is homeomorphic to a unital \sC algebra; (2) $\Sigma$ is {\em full}, i.e. for every $v \in \Sigma$ there is a section $a : X \to \Sigma$, $Q \circ a = id_X$, such that $a (x) = v$; (3) The algebraic operations $(+,\cdot,*)$ are continuous on each fibre $\Sigma_x$. 
The equivalence with the category of $C(X)$-algebras is realized by recognizing that the set $S_X(\Sigma)$ of sections of a \sC bundle $\Sigma$ is a $C(X)$-algebra; on the other side, every $C(X)$-algebra $\mA$ defines a \sC bundle $\wa \mA$ given by the disjoint union $\wa \mA :=$ $\dot{\cup}_x \mA_x$ endowed with a suitable topology, in such a way that $\mA$ is isomorphic to $S_X(\wa \mA)$. If $\eta : \mA \to \mA$ is a $C(X)$-morphism, then we denote the associated morphism of \sC bundles by $\wa \eta : \wa \mA \to \wa{\mA'}$; $\wa \eta$ is determined by the relations $\wa \eta \circ \pi_x (a) =$ $\pi'_x \circ \eta (a)$, where $a \in \mA$, $x \in X$, and $\pi'_x : \mA' \to \mA'_x$ is the fibre epimorphism.
In particular, continuous bundles of \sC algebras are $C(X)$-algebras; the associated \sC bundles are characterized by the property that the projection $Q$ is open.
For details on this topics, see \cite{HK79}.

\

If $\mA$ is a \sC algebra and $\mM$ a right Hilbert $\mA$-module (\cite{Bla}), then we denote the group of unitary, right $\mA$-module operators by $U \mM$, and the \sC algebra of compact, right $\mA$-module operators by $K(\mM)$. Moreover, we denote the internal tensor product of Hilbert $\mA$-bimodules by $\otimes_\mA$. 
In the present paper we shall make use of the following construction. Let $\phi : \mC \to \mB$ be a nondegenerate \sC morphism and $\mN$ a right Hilbert $\mC$-module. The algebraic tensor product $\mN \odot_{\mC} \mB$ with coefficients in $\mC$ is endowed with a natural $\mB$-valued scalar product; we denote the right Hilbert $\mB$-module obtained by the corresponding completion by $\mN \otimes_{\mC} \mB$. In the sequel, we shall apply this construction to the following cases: (1) $\mC =$ $C(X)$, and $\mB$ is a unital $C(X)$-algebra; (2) $\mC$ is a $C(X)$-algebra, and $\mB = \mC_x$, $x \in X$, is the image of $\mC$ with respect to a fibre epimorphism.

\

About {\em Cuntz-Pimsner algebras} (\cite{Pim97}) we follow the categorical approach of \cite{DPZ97}. We recall that given $d \in \bN$ the Cuntz algebra $\mO_d$ is defined as the universal \sC algebra generated by a set $\left\{ \psi_i \right\}_{i=1}^d$ of orthogonal isometries with total support the identity (\cite{Cun77}); for every closed subgroup $G \subseteq \ud$, there is an automorphic action
\begin{equation}
\label{eq_act_od}
G \to {\bf aut} \mO_d 
\ \ , \ \ 
g \mapsto \wa g
\ : \
\wa g (\psi_i) := \sum_j g_{ij} \psi_j \ ,
\end{equation}
\noindent where $\left\{ g_{ij} \right\}$ is the set of matrix coefficients of $g$ (see \cite{DR87}).

\section{Group bundles and Cuntz-Pimsner algebras.}
\label{key_group_duals}

In the present section we give a different version of some results proved in \cite{Vas04}. Instead of focusing our attention to (noncompact) section groups acting on \sC algebras, we consider the underlying group bundles.
In particular, we make use of the notion of {\em fibred action} of a group bundle on a $C(X)$-algebra introduced in \cite{Vas02g}. 
There are several reasons for this change of scenario. First, it was kept in evidence in a previous work that the relevant properties of the \sC dynamical systems $( \cog , \sigma_\mG )$ in which we are interested depend on a group bundle rather than the group itself (see \cite[Definition 4.3]{Vas04}). Secondly, fibred actions need less technicalities, and have been used to prove a result that shall play an important role in the present paper (\cite[Theorem 6.1]{Vas02g}). 
Our approach is also motivated by a recent work of V. Nistor and E. Troitsky: the notions of fibred action on a $C(X)$-algebra and gauge-equivariant Hilbert bimodule (in the sense of the following \S \ref{sec_fa_hb}) may be regarded as noncommutative versions of the {\em gauge actions} introduced in \cite[\S 3]{NT04}. For this reason, but also to emphasize the role that group bundles will play in the present paper, we will use the term {\em gauge action} instead of {\em fibred action}.

Let $X$ be a compact Hausdorff space.
A {\em group bundle} is given by a bundle $p : \mG \to X$ such that each fibre $G_x :=$ $p^{-1}(x)$ is homeomorphic to a locally compact group. In general, we do not assume that $\mG$ is locally trivial, thus the isomorphism class of the fibres may vary at varying of $x$ in $X$.
We note that the \sC algebra $C_0(\mG)$ is a $C(X)$-algebra in a natural way, and recall from \cite{Vas02g} that an {\em invariant $C(X)$-functional} is a positive $C(X)$-module map $\varphi : C_0(\mG) \to$ $C(X)$ such that $\varphi z (x) = \int_{G_x} z (y) d \mu_x(y)$, $z \in C_0(\mG)$, where $\mu_x$ is a Haar measure of $G_x$.
The set $S_X(\mG)$ of sections of $\mG$ is endowed with a natural structure of topological group. A {\em section group} $G$ is a subgroup of $S_X(\mG)$ such that for every $y \in \mG$ there is $g \in G$ with $y = g \circ p(y)$. For example, if $G_0$ is a locally compact group and $\mG$ is the trivial bundle $X \times G_0$, then the group $C(X,G_0)$ of continuous maps from $X$ to $G_0$ is a section group; anyway, also the group of constant $G_0$-valued maps is a section group for $\mG$.

Let $\mG$ be a group bundle, and $\mA$ is a $C(X)$-algebra. A {\em gauge action} of $\mG$ on $\mA$ is a family of strongly continuous actions $\left\{ \alpha^x : G_x \to {\bf aut} \mA_x \right\}_{x \in X}$ such that the map
\begin{equation}
\label{def_fa}
\alpha : \mG \times_X \wa \mA \to \wa \mA
\ \ , \ \ 
\alpha (y,v) := \alpha^x_y (v)
\end{equation}
\noindent is continuous. A $C(X)$-algebra endowed with a gauge action is called $\mG$-$C(X)$-{\em algebra}. The fixed-point algebra $\mA^\alpha$ is given by the \sC subalgebra of those $a \in \mA$ such that $\alpha^x_y \circ \pi_x(a) =$ $\pi_x(a)$ for all $y \in \mG$, $x := p(y)$. If $\mG$ is endowed with an invariant $C(X)$-functional $\varphi$, then an invariant mean $m_\varphi : \mA \to \mA^\alpha$ is naturally defined. For each section group $G \subseteq S_X(\mG)$ there is an action
\begin{equation}
\label{def_ind_a}
\alpha^G : G \to {\bf aut}_X \mA
\end{equation}
\noindent such that $\pi_x \circ \alpha^G_g(a) =$ $\alpha^x_{g(x)} \circ \pi_x(a)$ for every $a \in \mA$, $g \in G$, $x := p(y)$. In particular, usual strongly continuous actions $\alpha_0 : G_0 \to$ ${\bf aut}_X \mA$ correspond to gauge actions of the trivial bundles $X \times G_0$.

The following class of examples shall play an important role in the present paper. Let us consider a rank $d$ vector bundle $\mE \to X$ with associated $\ud$-cocycle $( \left\{ X_i \right\} , \left\{ u_{ij}  \right\} ) \in$ $H^1 (X,\ud)$ (see \cite[I.3.5]{Kar}). For the rest of the present paper, $\wE$ will denote the Hilbert $C(X)$-bimodule of sections of $\mE$, and $\left\{ \psi_l \right\}$ a (finite) set of generators of $\wE$. 
The Cuntz-Pimsner algebra $\coe$ associated with $\wE$ (see \cite{Pim97}) can be described in terms of generators and relations:
\begin{equation}
\label{def_cp}
\psi_l^* \psi_m = \left \langle \psi_l , \psi_m \right \rangle
\ \ , \ \ 
f \psi_l = \psi_l f
\ \ , \ \ 
\sum_l^n \psi_l \psi_l^* = 1 \ ,
\end{equation}
\noindent where $f \in C(X)$, $\left \langle \cdot , \cdot \right \rangle$ is the $C(X)$-valued scalar product of $\wE$, and $1$ is the identity. Let $L :=$ $\left\{ l_1 , \ldots , l_r \right\}$ be a multi-index of length $r \in \bN$ (in such a case, we write $|L| = r$); defining
\begin{equation}
\label{def_psil}
\psi_L := \psi_{l_1} \cdots \psi_{l_r} \ ,
\end{equation}
\noindent we obtain an element of $\coe$ which can be regarded as an element of the $C(X)$-bimodule tensor product $\wE^r$.
Now, $\coe$ is a continuous bundle over $X$ with fibre the Cuntz algebra $\mO_d$ (see \cite{Vas04,Vas05}). The \sC bundle $\wa \mO_\mE \to X$ may be described as the clutching of the trivial bundles $X_i \times \mO_d$ with respect to the transition maps $( \left\{ X_i \right\} , \left\{ \wa u_{ij} \right\}) \in$ $H^1 ( X , {\bf aut} \mO_d )$ defined using (\ref{eq_act_od}). 
Let 
\[
p : \mcUE \to X
\]
denote the group bundle of unitary endomorphisms of $\mE$ (see \cite[I.4.8(c)]{Kar}), and  $\mG \subseteq \mcUE$ a compact group bundle. Then there is a natural action
\begin{equation}
\label{def_fae}
\mG \times_X \mE \to \mE
\end{equation}
in the sense of \cite[\S 3]{NT04}, in fact each fibre $\mcUE_x \simeq \ud$ acts by unitary operators on the corresponding fibre $\mE_x \simeq \bC^d$.
By universality of the Cuntz-Pimsner algebra, the action (\ref{def_fae}) extends to a gauge action
\begin{equation}
\label{24}
\mG \times_X \wa \mO_\mE \to \wa \mO_\mE
\ \ , \ \
(y,\xi) \mapsto \wa y (\xi)
\ ,
\end{equation}
\noindent which, fiberwise, behaves like (\ref{eq_act_od}). If there is a section group $G \subseteq$ $S_X(\mG)$ then we have an automorphic action $G \to {\bf aut}_X \coe$. An important example is the bundle $\mcSUE$ of special unitary endomorphisms of $\mE$, which induces the strongly continuous action $\mSUE \to {\bf aut}_X \coe$, where 
\[
\mSUE := S_X(\mcSUE)
\]
(see \cite{Vas04}). We denote the fixed-point algebra of $\coe$ with respect to the gauge action (\ref{24}) by $\cog$.
For future reference, we also introduce the canonical endomorphism
\begin{equation}
\label{def_ecoe}
\sigma_\mE (t) := \sum_l \psi_l t \psi_l^* \ \ , \ \ t \in \coe \ \ ;
\end{equation}
which fulfilles the relations $( \ers ) = ( \sers )$, $r,s \in \bN$. If $\mG \subseteq \mcUE$ then $\sigma_\mE$ restricts to an endomorphism $\sigma_\mG \in$ ${\bf end}_X \cog$ such that $( \ers )_\mG =$ $( \sgrs )$, $r,s \in \bN$ (see \cite[\S 4]{Vas04}).

Let us now denote the tensor category with objects $\mE^r$, $r \in \bN$, identity object $\iota :=$ $\mE^0 :=$ $X \times \bC$, and arrows $(\ers)$, $r,s \in \bN$ by $\mE^\otimes$ (so that, $( \iota , \iota ) = C(X)$). By the Serre-Swan equivalence, $(\ers)$ is the set of $C(X)$-module operators from $\wE^r$ to $\wE^s$, and can be interpreted as the module of sections of a vector bundle $\mE^{r,s} \to X$ (see \cite[I.4.8(c)]{Kar}); in this way, every $t \in (\ers)$ can be regarded as a map $t : X \hra \mE^{rs}$. By using the notation (\ref{def_psil}), we also find
\[
(\ers) = 
{\mathrm{span}}
\left\{ 
f \psi_L \psi_M^*
 \ , \ 
f \in C(X) ,
|L| = s , |M| = r
\right\}
\ \ ;
\]
\noindent in fact, $\psi_L \psi_M^*$ can be naturally identified with the operator $\theta_{LM} \varphi :=$ $\psi_L \left \langle \psi_M , \varphi \right \rangle$, $\varphi \in$ $\wE^r$. Since each $\psi_L$ belongs to $\coe$, every $(\ers)$ can be regarded as a subspace of $\coe$.
Let us now consider the gauge action (\ref{def_fae}), and the spaces of invariant morphisms
\begin{equation}
\label{def_ersg}
(\ers)_\mG 
:=
\left\{
t \in (\ers) : \ y^s \cdot t(x) = t(x) \cdot y^r  \
\ , \ y \in \mG , x := p(y)
\right\}
\ ,
\end{equation}
\noindent where $y^r$ denotes the $r$-fold tensor power of $y$ as a linear operator on the fibre $\mE_x \simeq$ $\bC^d$. 
Then we can define a tensor category $\wa \mG$ with objects $\mE^r$, $r \in \bN$, and arrows $( \ers )_\mG$. Note that $(\ii)_\mG =$ $C(X)$, and that $\wa \mG$ is symmetric in the sense of \cite{DR89}; in fact, the symmetry operator $\theta \in ( \mE^2 , \mE^2 )$, $\theta \psi \psi' := \psi' \psi$, $\psi , \psi \in \wE$, belongs to $( \mE^2 , \mE^2 )_\mG$ for every $\mG \subseteq \mcUE$.

When $X$ reduces to a point, $\coe$ is the Cuntz algebra $\mO_d$ and $\wa G$ is the category of tensor powers of the defining representation of a compact Lie group $G \subseteq \ud$ (see \cite{DR87}).

The following Lemma \ref{lem_og} can be proved using a standard argument based on the mean $m_\varphi : \coe \to \cog$ induced by the invariant functional $\varphi : C(\mG) \to C(X)$, whilst the successive Lemma \ref{lem_str_og} is just a "fibred version" of \cite[Cor.4.9]{Vas04}; so that, in both the cases we omit the proof.

\begin{lem}
\label{lem_og}
Let $\varphi : C(\mG) \to C(X)$ be an invariant functional. Then the set $^0 \cog := \cup_{r,s} (\ers)_\mG$ is dense in $\cog$.
\end{lem}
%
%
%\begin{proof}
%If $t \in (\ers)_\mG$, then we find $t(x) =$ $y^s \cdot t(x) \cdot (y^r)^* =$ 
%$\wa y (t(x))$, thus $^0 \cog$ is contained in $\cog$. On the converse, we note that by
%construction the set $^0 \coe :=$ $\cup_{r,s} (\ers)$ is dense in $\coe$; thus, if 
%$t \in \cog$, then $t$ is norm-limit of a sequence $\left\{ t_n \right\}$ contained in 
%$^0 \coe$. For every $n \in \bN$, we assume for simplicity that $t \in (\ers)$, and apply
%the invariant mean $m_\varphi : \coe \to \cog$, in such a way that $m_\varphi t (x) :=$
%$\int_{G_x} \wa y (t(x)) d \mu_x(y)$ for every $x \in X$. By construction, 
%$m_\varphi t \in$ $(\ers)_\mG$, and the sequence $\left\{ m_\varphi t_n \right\}$
%approximates $t$.
%\end{proof}
%
%
\begin{lem}
\label{lem_str_og}
Let $\mG$, $\mG' \subseteq$ $\mcUE$ be compact group bundles, and $u \in S_X (\mcUE)$. If $\wa u \in {\bf aut}_X \coe$ restricts to an isomorphism from $\cog$ onto $\mO_{\mG'}$, then $\mG' = u \mG u^* :=$ $\left\{  u(x) G_x u(x)^* , x \in X  \right\}$.
On the converse, if $\mG' = u \mG u^*$, then $\wa u \in$ ${\bf aut}_X \coe$ restricts to an isomorphism $\mO_{\mG'} \to \cog$. 
\end{lem}

Now, $\cog$ is a continuous bundle over $X$ with fibres $(\cog)_x$, $x \in X$; we define
\[
\mSG :=
\left\{
u \in \mcUE : \wa u (t) = t \ , \ t \in (\cog)_{p(u)}
\right\} \ .
\]
It is clear that $\mG \subseteq \mSG$. The bundle $\mSG$ is called the {\em spectral bundle} associated with $\mG$, and may be regarded as a 'regularization' of $\mG$ (see the example below). For every $r,s \in \bN$ we have $( \ers )_\mSG = ( \ers )_\mG$ (see \cite[Lemma 4.10]{Vas04}); thus, at the level of the category $\wa \mG$ it is not possible to distinguish $\mG$ from $\mSG$. In the sequel of the present paper {\em we will always assume that} $\mG = \mSG$. As an example, take $X = [0,1]$, $\mE = X \times \bC^d$, $\mG =$ $\left\{ (x,u) \in X \times \sud : \right.$ $\left. x = 1/2 \Rightarrow u = 1 \right\}$; we find $\coe = C(X) \otimes \mO_d$ and $\cog =$ $C(X) \otimes \mO_\sud$, so that $\mG$ is strictly contained in $\mSG = X \times \sud$.

\section{Gauge-equivariant Hilbert bimodules.}
\label{sec_fa_hb}

In the present section we discuss some basic properties of the category of $C(X)$-Hilbert bimodules in the sense of Kasparov (\cite{Kas88}), and introduce the notion of gauge-equivariant Hilbert bimodule. This class of bimodules will yield the model category for the duality that we shall prove in \S \ref{duality}.
Since in the sequel we shall make use of unital \sC algebras, to simplify the exposition we discuss only the case in which the coefficient algebra of our bimodules is unital; the non-unital case will be approached in a future paper.

Let $\mA$, $\mB$ be unital $C(X)$-algebras. We denote the fibre epimorphisms of $\mA$, $\mB$ respectively by $\pi_x : \mA \to \mA_x$, $\pi'_x : \mB \to \mB_x$, $x \in X$, and the identity of $\mB$ by $1$.

A $C(X)$-{\em Hilbert} $\mA$-$\mB$-{\em bimodule} is given by a Hilbert $\mA$-$\mB$-bimodule such that $\psi f =$ $f \psi$, $f \in C(X)$, $\psi \in \mM$. 
We denote the category of $C(X)$-Hilbert $\mA$-$\mB$-bimodules by ${\bf bmod}_X ( \mA , \mB )$, with arrows the sets $( \mM , \mM' )$ of adjointable, (bounded) right $\mA$-module operators $T : \mM \to \mM'$, $\mM , \mM' \in$ ${\bf bmod}_X (\mA , \mB)$. In particular, the spaces of compact operators are denoted by $K(\mM,\mM')$.

By definition of $C(X)$-Hilbert bimodule, every $T \in$ $( \mM , \mM' )$ fulfilles the relations 
\[
T (f \psi) = T (\psi f) = T(\psi) f = f T(\psi)
\ \ , \ \
\psi \in \mM
\ , \
f \in C(X)
\ .
\]
This implies that $( \mM , \mM )$ is a $C(X)$-algebra, and the same is true for the ideal of compact operators $K(\mM) \subseteq$ $( \mM , \mM )$. The property of $\mM$ being a $C(X)$-Hilbert bimodule is translated as the fact that the left $\mA$-module action $\phi : \mA \to ( \mM , \mM )$ is a $C(X)$-morphism. To be concise, in the sequel we will write $a \psi \equiv \phi (a) \psi$, $a \in \mA$, $\psi \in \mM$; on the other side, the operator $\phi (a) \in (\mM , \mM)$ shall not be confused with $a \in \mA$.
If $\wa \mL_\mM \to X$ is the \sC bundle associated with $( \mM , \mM )$, then by general properties of $C(X)$-algebras there is a morphism
\begin{equation}
\label{eq_la_b}
\wa \phi : \wa \mA \to \wa \mL_\mM \ .
\end{equation}
If we denote the fibre epimorphisms of $( \mM , \mM )$ by
\begin{equation}
\label{def_fmm}
\delta_x : ( \mM , \mM ) \to \mL_{\mM,x} 
\ \ , \ \
x \in X
\ ,
\end{equation}
then $\wa \phi$ is determined by the relations $\wa \phi \circ \delta_x =$ $\delta_x \circ \phi$.
For every $x \in X$, we consider the right Hilbert $\mB_x$-module $\mM_x := \mM \otimes_\mB \mB_x$ and the associated map 
\begin{equation}
\label{def_emm}
\eta_x : \mM \to \mM_x 
\ , \
\eta_x (\psi) := \psi \otimes \eta_x(1) \ .
\end{equation}
Note that (\ref{def_emm}) is surjective: in fact $\psi \otimes w = \eta_x (\psi b)$, where $\psi \in \mM$, $w \in \mB_x$, and $b \in \pi_x^{-1} (w)$.

\begin{lem}
\label{lem_mx}
For every $x \in X$, the space $\mM_x$ has the following structure of Hilbert $\mA_x$-$\mB_x$-bimodule:
\[
\left\{
\begin{array}{ll}
(\psi \otimes w) w' := \psi \otimes (ww')
\\
\left \langle \psi \otimes w , \psi' \otimes w' \right \rangle_x
:=
\pi_x ( \left \langle \psi , \psi' \right \rangle ) \ w^* w'
\\
v (\psi \otimes w) := (a \psi) \otimes w  \ , \ \pi_x (a) = v
\ ,
\end{array}
\right.
\]
$\psi , \psi' \in \mM$, $v \in \wa \mA$, $w,w' \in \wa \mB$, $a \in \mA$. Moreover, for every $x \in X$ there is a natural isomorphism $\mL_{\mM,x} \simeq ( \mM_x , \mM_x )$.
\end{lem}

\begin{proof}
The fact that the right $\mB_x$-module action and the scalar product are well-defined follows by general properties of internal tensor products of Hilbert bimodules. Thus, it remains to verify only that the left $\mA_x$-module action is well-defined. To this end, let $a,a' \in \mA_x$ such that $v =$ $\pi_x(a) =$ $\pi_x (a')$. Then there are $f \in C_x(X)$, $a'' \in \mA$ such that $fa'' = a - a'$. In this way, we find $(a \psi) \otimes w =$ $[(a' + f a'') \psi] \otimes w =$ $(a' \psi) \otimes w +$ $f(x) (a'' \psi) \otimes w =$ $(a' \psi) \otimes w$. This implies that our definition of left $\mA_x$-module action does not depend on the choice of $a \in \pi_x^{-1} (v)$.
Finally, if $\delta_x$, $\eta_x$ are as by (\ref{def_fmm},\ref{def_emm}), then an isomorphism $\beta_x : \mL_{\mM,x} \to (\mM_x,\mM_x)$ is defined by 
\[
[\beta_x \circ \delta_x (t)] \ [\eta_x (\psi)]
\ := \
\eta_x (t \psi)
\ ,
\]
$t \in ( \mM , \mM )$, $\psi \in \mM$ (we leave to the reader the task to verify that $\beta_x$ is actually an isomorphism).
\end{proof}

As for $C(X)$-algebras, we endow the disjoint union $\wa \mM :=$ $\dot \cup_x \mM_x$ with the natural projection $P : \wa \mM \to X$, and the topology generated by the base
\begin{equation}
\label{def_top}
T_{U,\eps,\psi} :=
\left\{   
\xi \in P^{-1}(U) : \left\| \xi - \eta_{P(\xi)}(\psi)  \right\| < \eps
\right\}
\ , 
\end{equation}
where $U \subseteq X$ is open, $\psi \in \mM$, and $\eps > 0$. For the notion of {\em Banach bundle}, see \cite{Dup74,HK79} (anyway, it is analogous to the one of \sC bundle).

\begin{lem}
The map $P : \wa \mM \to X$ defines a full Banach bundle, and $\mM$ coincides with the set $S_X (\wa \mM)$ of sections of $\wa \mM$.
\end{lem}

\begin{proof}
By \cite[Theorem 6.2]{HK79}, to prove the Lemma it suffices to verify that $\mM$ is a locally convex Banach $C(X)$-bimodule, i.e. that for every $f \in C(X)$, $0 \leq f \leq 1$, $\psi_1 , \psi_2 \in \mM$, $\left\| \psi_1 \right\|$, $\left\| \psi_2 \right\| \leq 1$, it turns out $\left\| f \psi_1 + (1-f) \psi_2 \right\| \leq 1$. This can be easily done by using the $\mB$-valued scalar product, and estimating
\begin{equation}
\label{eq_ineq}
\left \langle
f \psi_1 + (1-f) \psi_2
\ , \
f \psi_1 + (1-f) \psi_2
\right \rangle
\ \leq \
f^2 + (1-f)^2
\leq 
1
\end{equation}
(note that $\left \langle f \psi_1 , (1-f) \psi_2 \right \rangle = 0$; for the last inequality, we used the relations $f^2 \leq f$, $(1-f)^2 \leq 1-f$). Taking the norms of the terms of (\ref{eq_ineq}), we conclude that $\mM$ is locally convex.
\end{proof}

The bundle $\wa \mM$ may be endowed with further structure: in fact, the $\mA$-$\mB$-bimodule structure and the $\mB$-valued scalar product of $\mM$ induce continuous maps
\[
\wa \mA \times_X \wa \mM \to \wa \mM
\ \ , \ \
\wa \mM \times_X \wa \mB \to \wa \mM
\ \ , \ \
\wa \mM \times_X \wa \mM \to \wa \mB
\ .
\]
The correspondence $\mM \mapsto \wa \mM$ established in the previous Lemma has a functorial nature. If $T \in ( \mM , \mM' )$, then there is an associated bundle morphism $\wa T : \wa \mM \to {\wa \mM}'$, determined by the property $\wa T \circ \eta_x (\psi) =$ $\eta'_x \circ T(\psi)$, where $\psi \in \mM$, and $\eta_x$, $\eta'_x$ denote the evaluations of $\mM$, $\mM'$ on $x \in X$.

\begin{ex}
\label{ex_ab_hb}
{\it
Let $q : \Omega \to X$ be a compact bundle (i.e., $C(\Omega)$ is a unital $C(X)$-algebra), and $\mE \to \Omega$ a vector bundle. For every $x \in X$, we consider the vector bundle obtained as the restriction $\mE |_{\Omega_x} \to \Omega_x$, $\Omega_x :=$ $q^{-1}(x)$. We denote the set of sections of $\mE$ by $\mM$; clearly, $\mM$ is a $C(X)$-Hilbert $C(\Omega)$-bimodule. 
A standard argument (the Tietze extension theorem for vector bundles, see \cite[1.6.3]{Ati}) allows one to conclude that for every $x \in X$ there is a natural isomorphism
\[
V : \mM_x := \mM \otimes_{C(\Omega)} C(\Omega_x) \to S_{\Omega_x} (\mE |_{\Omega_x})
\ \ , \ \
V(\psi \otimes z) := \psi |_{\Omega_x} z
\ ,
\]
where $\psi : \Omega \to \mE$ belongs to $\mM$, $z \in C(\Omega_x)$, and $S_{\Omega_x} (\mE |_{\Omega_x})$ is the Hilbert $C(\Omega_x)$-bimodule of sections of $\mE |_{\Omega_x}$. 
Thus, the bundle $\wa \mM \to X$ has fibres $S_{\Omega_x} (\mE |_{\Omega_x})$, $x \in X$.
}
\end{ex}

\begin{ex}
\label{ex_mf}
{\it
Let $\mA$ be a \sC algebra, $X$ a compact Hausdorff space, and $\mE \to X$ a vector $\mA$-bundle in the sense of \cite{MF80}. Then the set $\mM$ of sections of $\mE$ is a $C(X)$-Hilbert $C(X)$-$(C(X) \otimes \mA)$-bimodule, and $\wa \mM = \mE$.
}
\end{ex}

\begin{defn}
\label{def_fabm}
Let $p : \mG \to X$ denote a group bundle with gauge actions $\alpha : \mG \times_X \wa \mA \to$ $\wa \mA$, $\beta : \mG \times_X \wa \mB \to \wa \mB$, and $\mM$ a $C(X)$-Hilbert $\mA$-$\mB$-bimodule. A {\bf gauge action} of $\mG$ on $\mM$ is given by a family $\left\{ U^x : G_x \times \mM_x \to \mM_x \right\}_x$ of actions by isometric linear operators making each $\mM_x$ a $G_x$-Hilbert $\mA_x$-$\mB_x$-bimodule, such that the map
\[
U : \mG \times_X \wa \mM \to \wa \mM
\ \ , \ \ 
U(y,\xi) := U^x_y \xi
\ ,
\] 
is continuous. In such a case, we say that $\mM$ is a $\mG${\bf -Hilbert} $\mA$-$\mB$-{\bf bimodule}.
\end{defn}

In explicit terms, a $\mG$-Hilbert $\mA$-$\mB$-bimodule $\mM$ is characterized by the relations
\begin{equation}
\label{eq_fabm}
\left\{
\begin{array}{ll}
\left \langle U^x_y(\xi) , U^x_y(\xi') \right \rangle_x
=
\beta^x_y (  \left \langle \xi , \xi' \right \rangle_x)
\\
U^x_y \cdot \wa \phi (v) \ = \ \wa \phi \circ \alpha^x_y (v) \cdot U^x_y
\ ,
\end{array}
\right.
\end{equation}
$y \in \mG$, $x := p(y)$, $\xi , \xi' \in \wa \mM$, $v \in \wa \mA$. If $\mM$, $\mM'$ are $\mG$-Hilbert $\mA$-$\mB$-bimodules, we define the set of $\mG$-equivariant operators
\[
( \mM , \mM' )_\mG
:=
\left\{
T \in ( \mM , \mM' ) : \
\wa T \circ U ( y , \xi ) = U ( y , \wa T \xi ) 
\ , \ 
y \in \mG , \xi \in \wa \mM 
\right\}
\ .
\]
The proof of the following result is analogous to \cite[Proposition 3.3]{Vas02g}, thus it is omitted.
\begin{prop}
Let $\mM$ be a $C(X)$-Hilbert $\mA$-bimodule with a gauge action $U : \mG \times_X \wa \mM \to$ $\wa \mM$. Then for every section group $G \subseteq S_X(\mG)$ there is an action $U^G : G \times \mM \to \mM$, such that
\begin{equation}
\label{eq_fact}
\eta_x \circ U^G_g = U^x_{g(x)} \circ \eta_x \ , \ x \in X \ ,
\end{equation}
where $\eta_x$ is defined by (\ref{def_emm}). On the converse, every action $U^G : G \times \mM \to \mM$ such that there is a set $\left\{  U^x : G_x \times \mM_x \to \mM_x  \right\}_x$ fulfilling (\ref{eq_fact}) defines a gauge action $U$ that does not depend on $G$.
\end{prop}

In particular, if $\mG \simeq X \times G_0$ is a trivial bundle, then we may take $G = G_0$ and obtain a $G$-action on $\mM$ in the usual sense (\cite[VIII.20]{Bla}).

\begin{rem}
\label{rem_tac}
{\it
In the present paper we shall make use of actions $U : \mG \times_X \wa \mM \to \wa \mM$ coupled with trivial actions on the coefficient algebras, in such a way that
\begin{equation}
\label{eq_tr_fa}
\left\{
\begin{array}{ll}
\left \langle U^x_y(\xi) , U^x_y(\xi') \right \rangle_x
=
\left \langle \xi , \xi' \right \rangle_x 
\\ 
\left[ U^x_y  ,  \wa \phi(v) \right]  = 0
\ ,
\end{array}
\right.
\end{equation}
$y \in \mG$, $x := p(y)$, $\xi , \xi \in \wa \mM$, $v \in \wa \mA$. In this case, $U^x_y \in U \mM_x$, $x \in X$, and we may express (\ref{eq_tr_fa}(2)) in the following, more concise way: $\phi(a) \in ( \mM , \mM  )_\mG$, $a \in \mA$.
}
\end{rem}

The following Lemma generalizes (\ref{24}):
\begin{lem}
\label{lem_fa_om}
Let $\mM$ be a $\mG$-Hilbert $\mA$-bimodule with trivial gauge $\mG$-actions on $\mA$. Then the Cuntz-Pimsner algebra $\com$ is a $C(X)$-algebra with associated \sC bundle $\wa \mO_\mM$, and the $\mG$-action on $\mM$ extends to a gauge action
$\mG \times_X \wa \mO_\mM \to \wa \mO_\mM$.
\end{lem}

\begin{proof}
The Cuntz-Pimsner algebra $\com$ may be described as the one generated by the spaces $K( \mrs )$, $r,s \in \bN$, as in \cite[\S 4]{DR89}, \cite[\S 3]{DPZ97}. Since $f t = tf$, $t \in K(\mrs)$, $f \in C(X)$, we conclude that $C(X)$ is contained in the centre of $\com$, i.e. $\com$ is a $C(X)$-algebra. We denote the fibres of $\com$ by $(\com)_x$, $x \in X$, and the so-obtained \sC bundle by $\wa \mO_\mM \to X$. 
By \cite[\S 3]{DPZ97}, it follows that for every $x \in X$ there is a strongly continuous action $G_x \to {\bf aut} (\com)_x$, $y \mapsto \wa y$. In this way, we obtain a map $\mG \times_X \wa \mO_\mM \to \wa \mO_\mM$, $(y,\xi) \mapsto$ $\wa y (\xi)$, whose continuity can be easily proved using continuity of the action $U : \mG \times_X \wa \mM \to \wa \mM$.
\end{proof}

\begin{ex}[Equivariant vector bundles, \cite{Ati,Seg68}]
\label{ex_evb}
{\it
Let $G$ be a compact group, and $\Omega$ a compact Hausdorff $G$-space. We denote the orbit space by $X := \Omega / G$, and consider the natural projection $q : \Omega \to X$.
%
%A vector bundle $\mE \to \Omega$ is said to be $G$-equivariant if there is a continuous 
%action $G \times \mE \to \mE$ inducing unitary isomorphisms $g : \mE_\omega \to \mE_{g \omega}$,
%$\omega \in \Omega$. 
%
%
If $\mE \to \Omega$ is a $G$-equivariant vector bundle, then the Hilbert $C(\Omega)$-bimodule $\mM$ of sections of $\mE$ is naturally endowed with a $G$-action. 
%
%$V : G \to U \mM$, $V_g \psi (\omega) :=$ $g \psi (g^{-1}\omega)$, $\omega \in \Omega$, 
%$g \in G$, $\psi \in \mM$. 
%
In \sC algebraic terms, we have a strongly continuous action $\alpha : G \to C(\Omega)$ with fixed-point algebra (isomorphic to) $C(X)$, so that $C(\Omega)$ is a $C(X)$-algebra with associated \sC bundle $\wa \mA_\Omega \to X$ and fibres $\mA_{\Omega,x} \simeq$ $C(\Omega_x)$, $\Omega_x := q^{-1} (x)$, $x \in X$. As mentioned in \S \ref{key_group_duals}, $\alpha$ may be regarded as a gauge action $\mG \times_X \wa \mA_\Omega \to \wa \mA_\Omega$, where $\mG :=$ $X \times G$. 
Note that each $g \in G$ acts as a homeomorphism on the restriction $\Omega_x$, in fact $q(g \omega) = q(\omega)$ for every $\omega \in \Omega$; this implies that the $G$-action on $\mE$ restricts to $G$-actions $G \times \mE |_{\Omega_x} \to \mE |_{\Omega_x}$, $x \in X$.
By Example \ref{ex_ab_hb}, we find $\mM_x =$ $S_{\Omega_x} ( \mE |_{\Omega_x} )$, $x \in X$. Thus, we conclude that $\mM$ is a $\mG$-Hilbert $C(\Omega)$-bimodule.
%
%
%have actions $U^x : G \times \mM_x \to \mM_x$, $U^x_g \xi (\omega) :=$ $V_g \xi (\omega) =$ 
%$g \xi (g^{-1} \omega)$, $\xi \in \mM_x$, $\omega \in \Omega_x$, which induce a gauge action
%%
%\[
%U : \mG \times_X \wa \mM \to \wa \mM
%\ \ , \ \
%U ( (x,g) , \xi ) := U^x_g \xi
%\ .
%\]
%%
}
\end{ex}

\begin{ex}[Gauge-equivariant vector bundles, \cite{NT04}]
\label{ex_gevb}
{\it
Let $p : \mG \to X$ be a group bundle, and $q : \Omega \to X$ a compact bundle carrying an action $\mG \times_X \Omega \to \Omega$. In other terms, the \sC algebra $C(\Omega)$ is a $C(X)$-algebra endowed with a gauge action $\alpha : \mG \times_X \wa \mA_\Omega \to \wa \mA_\Omega$, where $\wa \mA_\Omega \to X$ is the \sC bundle associated with $C(\Omega)$.
Let $\pi : \mE \to \Omega$ be a vector bundle. As in the previous example, we have that the set $\mM$ of sections of $\mE$ is endowed with a structure of $C(X)$-Hilbert $C(\Omega)$-bimodule.
Note that $\mE$ is a bundle over $X$ with respect to the map $q \circ \pi$, with fibres $(q \circ \pi)^{-1}(x) =$ $\mE |_{\Omega_x}$, $x \in X$; thus, it makes sense to consider the fibred product $\mG \times_X \mE$. The vector bundle $\mE$ is said to be $\mG$-equivariant if there is an action by homeomorphisms
\[
\mG \times_X \mE \to \mE
\]
defined in such a way that each $\mE |_{\Omega_x} \to \Omega_x$, $x \in X$, is a $G_x$-equivariant vector bundle in the sense of the previous example. This implies that $\mM$ is a $\mG$-Hilbert $C(\Omega)$-bimodule. In the present paper we shall make use of the case in which $\Omega = X$ is endowed with the trivial $\mG$-action, as in \S \ref{key_group_duals}.
}
\end{ex}

A notion of Kasparov cycle in the setting of gauge-equivariant Hilbert bimodules can be introduced. Let $X$ be a compact Hausdorff space, $\mG \to X$ a group bundle, and $\mA$, $\mB$ unital, separable $\bZ_2$-graded $\mG$-$C(X)$-algebras. A {\em Kasparov cycle} is a pair $( \mM , F )$, where $\mM$ is a countably generated, $\bZ_2$-graded $\mG$-Hilbert $\mA$-$\mB$-bimodule with left $\mA$-action $\phi : \mA \to (\mM,\mM)$, and $F = F^* \in$ $( \mM , \mM )$ is an operator with degree one such that $[\phi(a),F]$, $[F^2 - 1 , \phi(a)] \in$ $K(\mM)$ for every $a \in \mA$; moreover, the following "quasi-$\mG$-equivariance" is required:
\[
[ \wa F , U^x_y ] \in K(\mM_x) \ \ , \ \ y \in \mG \ , \ x := p(y) \ .
\]
The notions of homotopy, equivalence and direct sum are exactly the same as in \cite{Kas88}. In this way, we can define the {\em gauge-equivariant} $KK$-{\em group} $KK^\mG_X ( \mA , \mB )$.
When $\mG = X \times G$ is trivial we obtain the usual Kasparov group $\mR KK^G ( X ; \mA , \mB )$ (\cite{Kas88}).
Further properties of the bifunctor $KK_X^\mG ( - , - )$ (in particular, details on the non-unital case) will be discussed in a future reference.

\begin{ex}
\label{ex_nt}
{\it
Let $\mG \to X$ be a compact Lie group bundle acting on Hilbert bundles $\mH^k \to X$, $k = 0,1$, and $F :=$ $\left\{ F^x : \mH^0_x \to \mH^1_x \right\}_x$ a continuous family of $\mG$-invariant Fredholm operators (see \cite[\S 4]{NT04}. Then the sets of sections $\wa \mH^k$, $k=0,1$, are $\mG$-Hilbert $C(X)$-bimodules, and $F$ may be regarded as an element of $( \wa \mH^0 , \wa \mH^1 )_\mG$. Let now $\mM := \wa \mH^0 \oplus \wa \mH^1$, and
\[
\widetilde F :=
\left(
\begin{array}{rl}
    0 & F^*
\\  F & 0
\end{array}
\right)
\ ;
\]
then, we obtain a pair $( \mM , \widetilde F )$, defining an element of $KK_X^\mG ( C(X) , C(X) )$. 
}
\end{ex}

Let $\mM , \mM' \in {\bf bmod}_X (\mA,\mB)$; we consider the space of right $\mB$-module operators that commute with the left $\mA$-action:
\begin{equation}
\label{def_bm}
( \mM , \mM' )_\mA
:=
\left\{
T \in ( \mM , \mM' )
\ : \
Ta = aT , a \in \mA
\right\}
\ .
\end{equation}
We now focalise our attention to the case $\mA = \mB$, and define $\mZ :=$ $\mA' \cap \mA$. The class ${\bf bmod}_X \mA :=$ ${\bf bmod}_X ( \mA , \mA )$ becomes a \sC category if equipped with the sets of arrows $(\mM , \mM' )$, $\mM$, $\mM' \in$ ${\bf bmod}_X \mA$.
Let us endow ${\bf bmod}_X \mA$ with the internal tensor product $\otimes_\mA$. It is well-known that at the level of arrows the tensor product $T \otimes_\mA T'$ does not make sense, unless $T'$ is the identity (\cite[13.5]{Bla}). In other terms, ${\bf bmod}_X \mA$ is a semitensor \sC category in the sense of \cite{DPZ97}. 
If we restrict the sets of arrows to $( \mM , \mM' )_\mA$, $\mM , \mM' \in {\bf bmod}_X \mA$, then we obtain a {\em tensor} \sC category ${\bf bmod}_{X,\mA} \mA$, having the same objects as ${\bf bmod}_X \mA$, and arrows $( \mM , \mM' )_\mA$. Note that every $( \mM , \mM' )_\mA$ is a Banach $\mZ$-bimodule in the natural way.

Let $\mG \to X$ denote a group bundle. We denote the semitensor \sC category with objects $\mG$-Hilbert $\mA$-bimodules by ${\bf bmod}_X^\mG \mA$, with arrows the spaces $( \mM , \mM' )_\mG$ of equivariant operators. 
Let $\mM$ be a $\mG$-Hilbert $\mA$-bimodule. The full semitensor \sC subcategory of ${\bf bmod}_X \mA$ with objects the internal tensor powers $\mM^r :=$ $\mM \otimes_\mA \ldots \otimes_\mA \mM$, $r \in \bN$, is denoted by $\mM^\otimes$; in particular, we define $\mM^0 := \mA$. Moreover, we denote the semitensor \sC category with objects $\mM^r$, $r \in \bN$, by $\mM_\mG^\otimes$, with arrows $(\mrs)_\mG$, $r,s \in \bN$; note that $( \mA , \mA )_\mG$ is the fixed-point algebra with respect to the given gauge action $\alpha : \mG \times_X \wa \mA \to \wa \mA$.

\begin{ex}
\label{ex_mez}
{\it
The following construction will be used in the sequel, in the context of Hilbert extensions.
Let $X$ be a compact Hausdorff space, $\mE \to X$ a vector bundle, and $\mZ$ an Abelian $C(X)$-algebra with identity $1$. We denote the bundle defined by the spectrum $X'$ of $\mZ$ by $q : X' \to X$.
Moreover, we consider a $C(X)$-Hilbert $\mZ$-bimodule $\mM$ with left action $\phi : \mZ \to$ $( \mM,\mM )$, and assume that there is an isomorphism $\mM \simeq$ $\wE \otimes_{C(X)} \mZ$ of right Hilbert $\mZ$-modules, where $\mE \to X$ is a rank $d$ vector bundle.
Now, $\mM$ is isomorphic (as a right Hilbert $\mZ$-module) to the module of sections of the pullback bundle $\mE^q :=$ $\mE \times_X X' \to X'$; note that elements of $\mE^q$ are of the type $(\xi,x')$, $\xi \in \mE_x$, $x \in X'_x$, $x \in X$. This allows one to regard elements of $\mM$ as injective continuous maps of the type $\psi : X' \to \mE$ such that $\psi (x') \in$ $\mE_{q(x')}$ (the associated section of $\mE^q$ is given by the map $\psi^q (x') :=$ $( \psi(x') , x' )$, $x' \in X'$).
On the same line of Example \ref{ex_gevb}, for every group bundle $p : \mG \to X$, $\mG \subseteq \mcUE$, we consider the gauge action
\begin{equation}
\label{eq_ga1}
\mG \times_X \mE^q \to \mE^q
\ \ , \ \
( y , (\xi,x') ) \mapsto ( y(\xi) , x'  )
\ ,
\end{equation}
obtained by extending (\ref{def_fae}) to $\mE^q$. Now, every $\mM_x$, $x \in X$, is isomorphic to the module of sections of $\mE^q |_{X'_x} \to X'_x$ (note that $\mE^q |_{X'_x} \simeq$ $\mE_x \times X'_x \simeq$ $\bC^d \times X'_x$); thus, elements of $\mM_x$ are continuous maps from $X'_x$ to $\mE_x$, and $\mM_x$ is isomorphic as a right Hilbert $\mZ_x$-module to the free module $\mE_x \otimes \mZ_x \simeq$ $\bC^d \otimes \mZ_x$. We obtain the gauge action
\begin{equation}
\label{eq_fam}
U : \mG \times_X \wa \mM \to \wa \mM
\ \ , \ \
U^x_y \psi (x') := y (\psi (x')) \ ,
\end{equation}
$y \in \mG$, $\psi \in \mM_{p(y)}$, $x' \in X'_{p(y)}$, and conclude that $\mM$ is a $\mG$-Hilbert $\mZ$-bimodule. 
Now, in general $\mM_\mG^\otimes$ is a semitensor \sC category; anyway, if
\begin{equation}
\label{eq_mg_t}
( \mrs )_\mG
\ \subseteq \ 
( \mrs )_\mZ
\end{equation} 
(see (\ref{def_bm})), then $\mM_\mG^\otimes$ is a tensor \sC category. It is clear that if the left and right $\mZ$-actions coincide, then (\ref{eq_mg_t}) is fulfilled. Anyway, (\ref{eq_mg_t}) holds in more general cases (see \cite[\S 4]{LV07}, and Example \ref{ex_cg}).
}
\end{ex}

\section{Dual actions vs. gauge actions.}
\label{sec_dual_action}

Let $\rho$ be a unital endomorphism of a \sC algebra $\mA$,
\begin{equation}
\label{def_zro}
\zro := \left\{ 
f \in \mA \cap \mA' : \rho (f) = f
\right\} \ ,
\end{equation}
\noindent $\mE \to X^\rho$ a vector bundle, $\mG \subseteq \mcSUE$ a compact group bundle. A {\em dual $\mG$-action} on $(\mA,\rho)$ is a $\zro$-monomorphism $\mu : ( \cog , \sigma_\mG ) \hra (\mA , \rho )$ such that $\mu (\ers)_\mG = \mu ( \sgrs ) \subseteq (\rhors)$, $r,s \in \bN$ (in other terms, $\mu$ is a functor from $\wa \mG$ to $\wa \rho$). The {\em crossed product} of $\mA$ by $\mu$ is the universal $\zro$-algebra $\mc$ generated by $\mA$ and an isomorphic image of $\coe$ by means of a unital $C(X)$-monomorphism $j : \coe \hra \mc$; moreover, we require that $\mc$ fulfills the universal properties
\begin{equation}
\label{30}
\left\{
\begin{array}{ll}
\mu (t) = j (t) \ \ , \ \  t \in \cog
\\
\rho (a) \psi = \psi a
\ \ , \ \
a \in \mA \ , \ \psi \in j(\wE) \subset  j(\coe) \ .
\end{array}
\right.
\end{equation}
The same argument of \cite[Thm 3.11]{Vas05} shows existence and unicity of $\mc$. Let us now denote the \sC bundle associated with $\mc$ by $\wa \mB \to X$; then, \cite[Eq.3.17]{Vas05} implies that there is a gauge action
\begin{equation}
\label{def_ada}
\alpha : \mG \times_X \wa \mB \to \wa \mB
\ \ , \ \ 
\alpha ( y , v \cdot \wa j (\xi) ) := 
v \cdot \wa j \circ \wa y (\xi) 
\ \ ,
\end{equation}
\noindent $v \in \wa \mA$, $\xi \in \wa \mO_\mE$. If $\mG$ is endowed with an invariant $C(X)$-valued functional, then $\mA$ is the fixed-point algebra with respect to the action (\ref{def_ada}) (see remarks after \cite[Theorem 3.11]{Vas05}, \cite[Lemma 2.8, Theorem 3.2]{DR89A}). If $\left\{ \psi_l \right\}$ is a finite set of generators for $j(\wE) \subset \mc$, the endomorphism
\begin{equation}
\label{def_sigmab}
\sigma \in {\bf end} (\mc)
\ \ , \ \
\sigma (b) := \sum_l \psi_l b \psi_l^*  \ , \ b \in \mc \ ,
\end{equation}
\noindent is defined, in such a way that $\sigma |_\mA =$ $\rho$, $\sigma \circ j =$ $j \circ \sigma_\mE$. As an immediate consequence of (\ref{def_sigmab}) and the remarks after (\ref{def_ecoe}), we find
\begin{equation}
\label{eq_erssrs}
j(\ers) \subseteq (\srs) \ \ , \ \ r,s \in \bN \ \ .
\end{equation}

\begin{lem}
\label{lem_funct}
Let $\mG \subseteq \mcSUE$ be a compact group bundle with a dual action $\mu : \wa \mG \to \wa \rho$. If $\mG' \subseteq \mcSUE$ is a compact group bundle with $\mG \subseteq \mG'$, then a dual $\mG'$-action $\mu'$ is defined on $(\mA,\rho)$, and there is a $C(X)$-epimorphism $\eta : \mA \rtimes_{\mu'} \wa{\mG}' \to \mc$. The morphism $\eta$ is an $\mA$-module map, intertwines the inner endomorphisms induced by $\wE$, and preserves the intertwiners spaces $(\ers)$, $r,s \in \bN$.
\end{lem}

\begin{proof}
Since $(\ers)_{\mG'} \subseteq (\ers)_\mG$, $r,s \in \bN$, we can define the dual action
\[
\mu' : (\ers)_{\mG'} \to (\rhors) 
\ \ , \ \
\mu'(t) := \mu (t)
\ \ ,
\]
and consider the crossed product $\mA \rtimes_{\mu'} \wa{\mG}'$, endowed with the \sC morphisms 
\[
\mu' : \mO_{\mG'} \to \mA \rtimes_{\mu'} \wa{\mG}'
\ \ , \ \ 
j \ ' : \coe \to \mA \rtimes_{\mu'} \wa{\mG}' \ .
\]
Now, it is clear that if we consider the generic element $a \cdot j(t)$ of $\mc$, $a \in \mA$, $t \in \coe$, then we find
\[
\left\{
\begin{array}{ll}
\mu' (t) = \mu (t) = j (t) \ \ , \ \ t \in \mO_{\mG'} \subseteq \cog
\\
\rho (a) \psi = \psi a \ \ , \ \ \psi \in j(\wE)
\end{array}
\right.
\]
This means that $\mc$ fulfilles the universal properties (\ref{30}) for the dual action $\mu'$, thus the desired morphism $\eta$ is defined by the universal property. Note that $\eta$ is defined as the quotient of $\mA \rtimes_{\mu'} \wa{\mG}'$ with respect to the ideal generated by the relations
$j \ '(t) =$ $\mu(t)$, $t \in \cog$
(recall that $\mu (t) \in \mA \subset$ $\mA \rtimes_{\mu'} \wa{\mG}'$, thus the above equality makes sense).
\end{proof}

Let $\mA \subset \mB$ be an inclusion of unital \sC algebras with common identity $1$. A {\em Hilbert} $\mA$-{\em bimodule in} $\mB$ is given by a norm-closed subspace $\mM \subset \mB$, closed with respect to left and right multiplication by elements of $\mA$, and such that the set $\left\{ \psi^* \psi , \psi,\psi' \in \mM \right\}$ coincides with $\mA$ (see \cite{DPZ97}). $\mM$ is said to have {\em support the identity} if there is a finite set $\left\{ \psi_l \right\} \subset \mM$ such that $\sum_l \psi_l \psi_l^* = 1$. Note that in this case, every $\psi \in \mM$ is of the form $\psi = \sum_l \psi_l a_l$, $a_l := \psi_l^* \psi \in$ $\mA$, so that $\mM$ is finitely generated. In the sequel, for every $r \in \bN$ we will denote the closed span of products of the type $\psi_1 \cdots \psi_r$, $\psi_i \in \mM$, by $\mM^r \subset \mB$. It is clear that $\mM$ may be identified with the $r$-fold internal tensor power of $\mM$. In the same way, the closed span of elements of the type $\psi' \psi^*$, $\psi' \in \mM^s$, $\psi \in \mM^r$, may be identified with the space $(\mrs)$ of right $\mA$-module operators from $\mM^r$ to $\mM^s$ (note that since $\mM$ is finitely generated, every element of $(\mrs)$ is a compact operator). 

A version of the following theorem has been proved in \cite[Proposition 3.12]{Vas05}; the proof that we give here presents the modifications due to the fact that we consider gauge actions instead of group actions.
\begin{thm}
\label{thm_dual0}
Let $\mG \subseteq \mcUE$ be a compact group bundle with a dual action $\mu : ( \cog , \sigma_\mG ) \to$ $( \mA , \rho )$. If $\mA' \cap (\mc) =$ $\mZ$, then the following properties hold:
\begin{enumerate}
\item There is a finitely generated $\zro$-Hilbert $\mZ$-bimodule $\mM \subset \mB$ with support $1$, isomorphic to $\wE \otimes_{C(X^\rho)} \mZ$ as a right Hilbert $\mZ$-module;
\item The gauge action $\alpha$ (see (\ref{def_ada})) restricts to a gauge action $U : \mG \times_X \wa \mM \to\wa \mM$;
\item $( \mrs )_\mG =$ $( \rhors )$, $\forall r,s \in \bN$.
\end{enumerate}
\end{thm}

\begin{proof}
Point 1: define
\[
\mM :=
\left\{
\psi \in \mc : \psi a = \rho (a) \psi , a \in \mA
\right\} \ ;
\]
since $\psi^* \psi' \in \mA' \cap (\mc)$ $= \mZ$, and $f \psi = \psi f$, $f \in \zro$, $\psi, \psi' \in \mM$, we find that $\mM$ is a $\zro$-Hilbert $\mZ$-bimodule in $\mc$. Moreover, (\ref{30}) implies that $j(\wE) \subset \mM$. Let now $\left\{ \psi_l \right\}$ be a finite set of generators for $j(\wE) \simeq \wE$. Since $j$ is unital, the relations (\ref{def_cp}) imply that $\mM$ has support the identity in $\mc$, and that $\mM$ is generated by elements of $j(\wE)$ as a right $\mZ$-module; this implies the isomorphism $\mM \simeq \wE \otimes_{C(X)} \mZ$. 
Point 2: let us consider the restrictions $\eta_x : \mM \to \mM_x :=$ $\eta_x (\mM)$ of the fibre epimorphism, and $v \in \mM_x$. By local triviality of $\mE$ there are sections $\varphi_1$, $\ldots$, $\varphi_d \in$ $j(\wE) \simeq$ $\wE$ fulfilling the Cuntz relations 
\[
\eta_x (\varphi_i^* \varphi_j) = \delta_{ij} \eta_x(1)
\ \ , \ \ 
\sum_i \eta_x (\varphi_i \varphi_i^*) = \eta_x(1) 
\ \ .
\]
In this way, $v =$ $\sum_j e_{j,x} \eta_x (z_j)$, where $z_j \in \mZ$ and $e_{j,x} :=$ $\eta_x(\varphi_j)$. By definition of $\alpha$, and by (\ref{eq_act_od}), for every $y \in G_x$ we find 
\[
\alpha (  y , v ) = \sum_{k,j} y_{kj} e_{j,x} \eta_x (z_j)
\]
where $y_{kj} \in \bC$ are the matrix coefficients of $y$. The above equality implies that $\alpha ( y,v ) \in \mM_x$, thus $\mM$ is stable with respect to the gauge action $\alpha$. We denote the restriction of $\alpha$ to $\mM$ by $U$. Since $\alpha^x_y \circ \eta_x (\psi^* \psi') =$ $\eta_x (\psi^* \psi')$ (in fact $\psi^* \psi' \in \mZ$), and $\alpha^x_y \circ \eta_x (\psi z) = $ $\alpha^x_y \circ \eta_x (\psi) \cdot \eta_x(z)$, $z \in \mZ$, we conclude that $U$ is actually a unitary gauge action in the sense of Definition \ref{def_fabm}, with trivial associated gauge actions on $\mZ$.
Point 3: if $t \in (\mrs)_\mG$ then by definition of $U$ we find $t \in \mA$. Moreover, since $t$ is sum of terms of the type $\psi' \psi^*$, $\psi \in \mM^r$, $\psi' \in \mM^s$, by definition of $\mM$ we conclude that $t \in (\rhors)$ (note in fact that $\psi \in \mM^r$ implies $\rho^r(a) \psi =$ $\psi a$, $a \in \mA$). On the converse, let $t \in (\rhors)$; we fix finite sets of generators $\left\{ \psi_L \right\} \subset \mM^s$, $\left\{ \varphi_M \right\} \subset \mM^r$, in such a way that $\sum_L \psi_L \psi_L^* =$ $\sum_M \varphi_M \varphi_M^* =$ $1$, and note that $t_{LM} := \psi_L^* t \varphi_M \in$ $\mA' \cap (\mc) =$ $\mZ$ (in fact, $t_{LM} a =$ $\psi_L^* t \rho^r(a) \varphi_M =$ $\psi_L^* \rho^s(a) t \varphi_M =$ $a t_{LM}$, $a \in \mA$). In this way, $t = \sum_{LM} \psi_L t_{LM} \psi_M^*$ belongs to $(\mrs)_\mG$.
%
%
%
%Point 4: as first, we note that (\ref{def_sigmab}) implies $j(\ers) \subseteq$ $(\srs)$,
%$r,s \in \bN$. Let $\beta \in$ ${\bf aut}_{\mA,\sigma} (\mc)$; if $\psi, \psi' \in$ 
%$j(\wE) \subset \mM$, then for every $b \in \mB$ we find
%%
%\[
%\psi^* \beta (\psi') b = 
%\psi^* \beta ( \sigma(b) \psi' ) =
%b \psi^* \beta (\psi') 
%\ \ .
%\]
%%
%This implies that $\psi^* \beta (\psi')$ belongs to the centre of $\mc$ and that 
%$\sigma ( \psi^* \beta (\psi') ) =$ $\psi^* \beta (\psi')$. Moreover, since the centre of
%$\mc$ is contained in $\mA' \cap (\mc) =$ $\mZ$, we conclude that $\psi^* \beta (\psi')$
%belongs to $\mZ$. This also implies $\psi^* \beta (\psi') \in$ $\zro$. Thus, we conclude
%that %
%\[
%\beta (\psi') = \sum_l \psi_l ( \psi_l^* \beta(\psi') ) \in j(\wE)
%\ \ , \ \
%\psi' \in \wE
%\ \ .
%\]
%%
%The previous equality implies that $\beta |_\wE$ defines a unitary morphism $u \in$
%$U(\wE)$, so that $\beta \circ j = j \circ \wa u$, where $\wa u$ is defined in the same way
%as in (\ref{eq_act_ud}). Now, it is well-known that every $u \in U(\wE)$ defines a section
%of the unitary bundle $\mcUE$. Moreover, since $\beta |_\mA = id_\mA$, we find that $\beta
%\circ j (t) =$ $j(t)$ for every $t \in \cog$, thus ([]) implies $u(x) \in \mG$ for every
%$x \in X$. By defining $y := u(x) \in$ $\mG$, we conclude $\eta_x \circ \beta = $
%$\eta_x \circ j []$
\end{proof}

\section{Special Endomorphisms.}
\label{cp_spec}

\subsection{Permutation symmetry: weak forms.}
\label{key_ps}

Let $\mA$ be a unital \sC algebra with centre $\mZ$. A unital endomorphism $\rho$ of $\mA$ has {\em permutation symmetry} if there is a unitary representation $p \mapsto \eps (p)$ of the group $\bP_\infty$ of finite permutations of $\bN$ in $\mA$, such that:
\begin{equation}
\label{eq_ps}
\left\{
\begin{array}{ll}
\eps (\bS p) = \rho \circ \eps (p)     \\
\eps := \eps (1,1) \in (\rho^2 , \rho^2)     \\
\eps (s,1) \ t = \rho (t) \ \eps (r,1) \  ,  \ t \in (\rhors)
\ ,
\end{array}
\right.
\end{equation}
\noindent where $(r,s) \in \bP_{r+s}$ exchanges the first $r$ terms with the remaining $s$, and $\bS$ is the shift $(\bS p)(1) := 1$, $(\bS p)(n) := 1 + p(n-1)$, $p \in \bP_\infty$. The above properties imply that
\begin{equation}
\label{weak_perm}
\eps(p) \in (\rho^n , \rho^n) \ \ , \ \ n \in \bN \ , \ p \in \bP_n \ ;
\end{equation}
\noindent we say that $\rho$ has {\em weak permutation symmetry} if just (\ref{eq_ps}(1)),(\ref{weak_perm}) hold. In such a case, we call {\em symmetry intertwiners} the elements of $(\rhors)$ for which (\ref{eq_ps}(3)) holds:
\[
( \rhors )_\eps := 
\left\{ t \in (\rhors) : \eps(s,1) t = \rho(t) \eps (r,1) \right\} \ .
\]
\noindent We denote the \sC subalgebra of $\mA$ generated by the intertwiner spaces $(\rhors)$, $r,s \in \bN$, by $\oro$; in particular, we denote the \sC algebra generated by the symmetry intertwiners by $\soro$. We say that $\rho$ has {\em permutation quasi-symmetry} if 
\begin{equation}
\label{gps3}
(\rhors) = \rho^s (\mZ) \cdot (\rhors)_\eps = (\rhors)_\eps \cdot \rho^r (\mZ) \ . 
\end{equation}
\noindent Here $\rho^s (\mZ) \cdot (\rhors)_\eps$ denotes the vector space spanned by elements of the type $\rho^s(z)t$, $z \in \mZ$, $t \in (\rhors)_\eps$. From (\ref{gps3}), we see that if $\rho$ has permutation quasi-symmetry, then the obstacle to get permutation symmetry is given by elements of $\mZ$ that are not $\rho$-invariant. 
The above properties may be summarized by saying that the category $\wa \rho$ with objects $\rho^r$, $r \in \bN$, and arrows $(\rhors)$ is a tensor quasi-symmetric \sC category, whilst the category $\wa \rho_\eps$ having the same objects as $\wa \rho$ and arrows $(\rhors)_\eps$ is a tensor symmetric \sC subcategory of $\wa \rho$ (see \cite[\S 4.1]{Vas05}).

\begin{rem}
\label{rem_wps}
{\it
By \cite[(2.6)]{DR87}, to get a weak permutation symmetry for $\rho$ it suffices that there is $\eps \in$ $( \rho^2 , \rho^2 )$ such that $\eps = \eps^* = \eps^2 = 1$. In fact, each $\eps (p)$, $p \in \bP_\infty$, may be recovered in terms of products of the type $\eps \rho (\eps) \ldots \rho^n (\eps)$, $n \in \bN$.
}
\end{rem}

If $\mE \to X$ is a vector bundle and $\mG \subseteq \mcUE$ is a compact group bundle, then $\sigma_\mG \in$ ${\bf end}_X \cog$ has permutation symmetry, induced by the flip operator $\theta \in$ $(\mE^2,\mE^2)_\mG =$ $(\sigma_\mG^2,\sigma_\mG^2)$ (see \cite[Remark 4.5]{Vas04}).

Weak forms of permutation symmetry are strictly related to dual actions:
\begin{lem}
\label{lem_da_dual}
Let $( \mA , \rho )$ be a \sC dynamical system, $\mE \to X^\rho$ a vector bundle, and $\mG \subseteq \mE$ a compact group bundle with a dual action $\mu : ( \cog , \sigma_\mG ) \to$ $( \mA,\rho )$. Then $\rho$ has a weak permutation symmetry $\eps$.
Moreover, if $\mA' \cap (\mc) =$ $\mZ$, then for every $r,s \in \bN$ the map $\mu$ restricts to an isomorphism $(\ers)_\mG \to (\rhors)_\eps$ of Banach $\zro$-bimodules.
\end{lem}

\begin{proof}
Since $\theta (r,s) \in$ $( \sigma_\mG^{r+s} , \sigma_\mG^{r+s} )$, $r,s \in \bN$, defining $\eps (r,s) := \mu \circ \theta (r,s)$ we obtain a weak permutation symmetry for $\rho$.
Now, if $t \in (\rhors)_\eps$ then using the decomposition $t =$ $\sum_{LM} \psi_L t_{LM} \psi_M^*$, $t_{LM} :=$ $\psi_L^* t \psi_M$ (where $\psi_L \in$ $j(\wE^s)$, see (\ref{def_psil})), in the same way as in the proof of Theorem \ref{thm_dual0} we find $t_{LM} \in$ $\mA' \cap (\mc) =$ $\mZ$. Moreover, we also find
\[
\rho (t_{LM})
=
\psi_L^* \eps (1,s) \eps (s,1) t \eps (1,r) \eps (r,1) \psi_M
=
t_{LM}
\ .
\]
This implies $t_{LM} \in$ $\zro$, so that $t \in j(\ers)$. Since $j(\ers) \cap \mA =$ $\mu (\ers)_\mG$, we conclude that $t \in \mu (\ers)_\mG$. On the converse, the fact that $\sigma_\mG \in$ ${\bf end}_{X^\rho} \cog$ has permutation symmetry, and the definition of $\eps$, imply that $\mu(\sgrs) =$ $\mu( \ers )_\mG \subseteq$ $( \rhors )_\eps$, and this proves that $(\rhors)_\eps =$ $\mu ( \ers )_\mG$.
\end{proof}

\subsection{Special endomorphisms and Crossed Products.}
\label{sec_spe}

The notion of {\em special conjugate} may be regarded as an algebraic counterpart of the property of a finite-dimensional group representation of being with determinant one, and plays a crucial role in the setting of the Doplicher-Roberts theory (\cite[\S 4]{DR89A},\cite[p.58]{DR90}). To take into account the class of examples in \S \ref{key_group_duals}, we generalize the definition of special conjugate in the following sense: we say that $\rho$ is {\em (weakly/quasi) special} if $\rho$ has (weak/quasi) permutation symmetry, and for some $d \in \bN$, $d > 1$, there is a Hilbert $\zro$-bimodule $\mR \subset ( \iota,\rho^d )_\eps$ such that
\begin{equation}
\label{scp2}
\left\{
\begin{array}{ll}
R^* \rho (R') \ = \ (-1)^{d-1} d^{-1} R^* R' \ , \ R,R' \in \mR      \\
{\mR}{\mR}^*  \ := \ 
{\mathrm{span}}\left\{ R' R^*  :  R , R' \in \mR \right\} \ = \ 
\zro P_{\eps,d} 
\end{array}
\right.
\end{equation}
\noindent When $\mR$ exists, it is unique (\cite[Lemma 4.8]{Vas05}). Let $K(\mR)$ denote the \sC algebra of right $\zro$-module operator of $\mR$; then, (\ref{scp2}(2)) implies that there is an isomorphism $K(\mR) \simeq \zro$. The previous isomorphism, and the Serre-Swan theorem, imply that $\mR$ is the module of sections of a line bundle over $X^\rho$; thus, the first Chern class $c_1 (\rho) \in H^2 ( X^\rho , \bZ )$ of such a line bundle is a complete invariant of $\mR$. The Chern class $c_1(\rho)$ vanishes if and only if $\rho$ fulfilles the special conjugate property in the sense of Doplicher and Roberts (\cite[Lemma 4.2]{DR89A}); this is equivalent to require that there is a partial isometry $S \in \mR$, $S^*S = 1$, $SS^* = P_{\eps,d}$. The {\em class} of $\rho$ is defined by
\begin{equation}
\label{def_class}
d \oplus c_1(\rho) \in 
\bN \oplus H^2 ( X^\rho , \bZ )
\ .
\end{equation}
\noindent A weakly special endomorphism is denoted by $( \rho , \eps , \mR )$.

A class of examples follow. With the notation of \S \ref{key_group_duals}, if $\mG \subseteq \mcSUE$ then $( \sigma_\mG , \theta , (\iota,\lambda \mE) )$ $\in {\bf end}_X \cog$ is special, where $\lambda \mE \subset \mE^d$ is the $d$-fold exterior power of $\mE$; moreover, $c(\sigma_\mG) = d \oplus c_1(\mE)$, where $d$ is the rank of $\mE$, and $c_1(\mE)$ is the first Chern class (see \cite[\S 2.2,Example 4.2]{Vas05}).

The following basic result has been proved in \cite[Theorem 5.1]{Vas05}.

\begin{thm}
\label{sue_action}
Let $( \rho , \eps , \mR ) \in$ ${\bf end} \mA$ be a weakly special endomorphism with class $d \oplus c_1(\rho)$, $d \in \bN$, $c_1(\rho) \in$ $H^2(X^\rho,\bZ)$. Then, for every rank $d$ vector bundle $\mE \to X^\rho$ with first Chern class $c_1(\rho)$, there is a dual $\mcSUE$-action 
\[
\mu : \mO_\mcSUE \to \mA \ .
\]
Moreover, $\soro$ is a continuous bundle of \sC algebras over $X^\rho$, with fibres isomorphic to $\mO_{G^x}$, $x \in X^\rho$, where $G^x \subseteq \sud$ is a compact Lie group unique up to conjugacy in $\sud$. Finally, there is an inclusion $\mO_{\mcSUE} \hra \soro$ of \sC algebra bundles.
\end{thm}

The image of $\mu$ is the \sC algebra generated by $\eps$, $\mR$ by closing with respect to the action of $\rho$. We have $\mu (\theta) = \eps$, $\mu (\iota , \lambda \mE)_\mcSUE = \mR$ (see \cite[Lemma 4.6]{Vas05}). For every rank $d$ vector bundle $\mE \to X^\rho$ with first Chern class $c_1(\rho)$, we introduce the notation
\begin{equation}
\label{def_be}
\mB_\mE := \mA \rtimes_\mu \wa \mcSUE \ .
\end{equation}
By (\ref{def_sigmab}), $\mB_\mE$ comes equipped with the endomorphism $\sigma \in {\bf end} \mB_\mE$. By definition of crossed product by a dual action, there are $\zro$-morphisms
\begin{equation}
\label{def_j}
\left\{
\begin{array}{ll}
j : ( \coe , \sigma_\mE ) \hra ( \mB_\mE , \sigma )
\\
\mu : (  \mO_\mcSUE , \sigma_\mcSUE ) \hra ( \mA , \rho )
\ , \
j |_{\mO_\mcSUE} = \mu
\ .
\end{array}
\right.
\end{equation}
Note that $\sigma$ has weak permutation symmetry induced by $\eps$. Since $\sigma_\mE \in$ ${\bf end}_{X^\rho} \coe$ has symmetry $\theta$, we find
%
%(see remarks after (\ref{def_class}), but also (\ref{eq_pssm}) below for an explicit 
%definition of the symmetry operators), we find
%
\begin{equation}
\label{eq_jers}
j ( \ers ) 
\ = \
j ( \sers )_\theta 
\ \subseteq \
( \srs )_\eps 
\ .
\end{equation}
Moreover, there is a gauge action
\begin{equation}
\label{def_asue}
\alpha : \mcSUE \times_X \wa \mB_\mE \to \wa \mB_\mE \ ,
\end{equation}
so that it is of interest to investigate some technical properties of the bundle $p : \mcSUE \to X^\rho$:
\begin{lem}
\label{lem_sue}
Let $X$ be a compact Hausdorff space, $d \in \bN$, and $\mE \to X$ a rank $d$ vector bundle. Then there is an invariant $C(X)$-functional $\varphi : C(\mcSUE) \to C(X)$. Moreover, for every group bundle $\mG \subseteq \mcSUE$ the natural projection $p_\mG : \mcSUE \to \Omega :=$ $\mG \backslash \mcSUE$ has local sections: in other terms, for every $\omega \in \Omega$ there is an open neighbourhood $W \ni \omega$ with a continuous map $s_W : W \to \mcSUE$ such that $p_\mG \circ s_W = id_W$.
\end{lem}

\begin{proof}
For every $u \in \ud$, we consider the group automorphism $\wa u \in {\bf aut} \sud$ defined by adjoint action $\wa u (g) :=$ $ugu^*$, $g \in \sud$; we maintain the same notation for the \sC algebra automorphism induced on $C(\sud)$, so that $\wa u \in$ ${\bf aut} C(\sud)$. If $( \left\{ X_i \right\} , \left\{ u_{ij} \right\} )$ is an $\ud$-cocycle associated with $\mE$ (see \cite{Kar}), then $\mcSUE$ is defined as the bundle with cocycle $( \left\{ X_i \right\} , \left\{ \wa u_{ij} \right\} )$. Let now $\lambda_{ij} (x) :=$ $\det u_{ij} (x)$, $x \in X_{ij}$, and $v_{ij} (x) := \ovl \lambda_{ij} (x) u_{ij} (x)$, $x \in X_{ij}$; in this way , for every pair $i,j$ a continuous map 
\[
v_{ij} : X_{ij} \to \sud
\]
is defined. Since $[\wa u_{ij} (x)] (g) =$ $[\wa v_{ij} (x)] (g)$, $g \in \sud$, we conclude that the bundle $\mcSUE$ has an associated $\sud$-cocycle $( \left\{ X_i \right\} , \left\{ \wa v_{ij} \right\} )$.
Let now $\varphi_0 : C(\sud) \to \bC$ denote the Haar measure of $\sud$; since $\sud$ is unimodular, for every $v \in \sud$ we find $\varphi_0 \circ \wa v =$ $\varphi_0$. Thus, applying \cite[Proposition 4.3]{Vas02g}, we conclude that the desired functional $\varphi : C(\mcSUE) \to C(X)$ exists.
About existence of local sections, we note that $\mcSUE$ is a locally trivial bundle with fibre the compact Lie group $\sud$, and apply \cite[Lemma 2.5]{Vas02g}.
\end{proof}

\begin{rem}
\label{rem_afp}
{\it
Existence of the invariant $C(X)$-functional $\varphi : C(\mcSUE) \to C(X)$ allows one to define an invariant mean $m_\varphi : \mB_\mE \to \mA$; using $m_\varphi$ we can conclude that $\mA$ is the fixed-point algebra of $\mB_\mE$ with respect to the action (\ref{def_asue}).
}
\end{rem}

We can now investigate the algebraic structure of the crossed product $\mB_\mE$.
Let us denote the centre of $\mB_\mE$ by $\mC_\mE$ and the spectrum of $\mC_\mE$ by $\Omega_\mE$. If $f \in \zro$, then $\psi f = \rho(f) \psi = f \psi$, $\psi \in j(\wE)$, thus $f$ commutes with $\mA$ and $j(\coe)$; this implies $f \in \mC_\mE$, so that there is a unital inclusion $\zro \hra \mC_\mE$ (i.e., $\mC_\mE$ is a $\zro$-algebra). We consider the surjective map defined by the Gel'fand transform
\begin{equation}
\label{def_ome_0}
q : \Omega_\mE \to X^\rho \ .
\end{equation}
The gauge action (\ref{def_asue}) restricts to a gauge action on $\mC_\mE$; by Gel'fand duality, there is a right action 
\begin{equation}
\label{def_aomega}
\alpha^* : \Omega_\mE \times_X \mcSUE \to \Omega_\mE
\ \ , \ \
\alpha^* (\omega,y) := \omega \circ \alpha^x_y \ , \ x = p(y) = q(\omega) \ ,
\end{equation}
in the sense of \cite[\S 3]{NT04}.

\begin{rem}
\label{rem_c}
{\it
$\mC_\mE$ is a continuous bundle of \sC algebras on $X$; for each $x \in X$, there is a closed group $G_x \subseteq \sud$ such that the fibre $\mC_{\mE,x}$ is isomorphic to $C(G_x \backslash \sud)$, i.e. $\Omega_x :=$ $q^{-1}(x)$ is homeomorphic to $G_x \backslash \sud$ (see \cite[Lemma 5.1]{Vas02g}).
}
\end{rem}

\begin{lem}
\label{lem_cp_str}
If $\rho$ has permutation quasi-symmetry, then the following properties hold:
\begin{enumerate}
\item  $\mA' \cap \mB_\mE =$ $\mC_\mE \vee \mZ$, where $\mC_\mE \vee \mZ$ denotes the \sC algebra generated by $\mC_\mE$, $\mZ$.
\item  $\mC_\mE$ is generated by elements of the type $t^* \varphi$, where $t \in (\iota , \rho^r)_\eps$, $\varphi \in j(\iota , \mE^r)$, $r \in \bN$.
\end{enumerate}
\end{lem}

\begin{proof}
Point 1: As a first step, we note that if $\mE$ is trivial then the result follows by \cite[Proposition 7.1]{Vas02g}. Let us now consider the general case in which $\mE$ is nontrivial; it is convenient to regard $\mB_\mE$, $\mA$ as $\zro$-algebras. Let $x \in X^\rho$, $U$ an open neighbourhood of $x$ trivializing $\mE$. We denote the restrictions of $\mB_\mE$, $\mA$ over $U$ by $\mB_{\mE,U}$, $\mA_U$ (see \S \ref{preli}); note that there are obvious inclusions
\begin{equation}
\label{eq_in}
\mA'_U \cap \mA_U \subseteq \mZ
\ \ , \ \ 
\mB'_{\mE,U} \cap \mB_{\mE,U} \subseteq \mC_\mE
\ \ .
\end{equation}
Now, $\mcSUE |_U \simeq U \times \sud$, and \cite[Proposition 3.3]{Vas02g} implies that there is a strongly continuous action
\[
\alpha_U : \sud \to \mB_{\mE,U}
\]
(in fact, $\sud$ may be regarded as a group of constant sections spanning $U \times \sud$). It is clear that $\sud$ acts on $\mB_{\mE,U}$ in such a way that the fixed point algebra $\mB_{\mE,U}^{\alpha_U}$ coincides with $\mA_U$. Thus, the argument used for $\mE \simeq X^\rho \times \bC^d$ implies that $\mA'_U \cap \mB_{\mE,U}$ is generated as a \sC algebra by $\mA'_U \cap \mA_U$ and $\mB'_{\mE,U} \cap \mB_{\mE,U}$; by (\ref{eq_in}), we conclude
\[
\mA'_U \cap \mB_{\mE,U} \subseteq \mC_\mE \vee \mZ \ \ .
\]
We now pick a finite open cover $\left\{ U_k \right\}$ trivializing $\mE$ with a subordinate partition of unity $\left\{ \lambda_k \right\} \subset$ $\zro$. If $b \in \mA' \cap \mB_\mE$, then 
\[
b = \sum_l \lambda_k b 
\ \ , \ \ 
\lambda_k b \in \mA'_{U_k} \cap \mB_{\mE,U_k}
\ \subseteq \ 
\mC_\mE \vee \mZ
\ \ .
\]
Point 2: let $t \in (\iota , \rho^r)_\eps$, $\varphi \in j(\iota , \mE^r)$. Then (\ref{eq_erssrs}) implies
\[
t^* \varphi a = t^* \rho^r(a) \varphi = a t^* \varphi \ \ , \ \ a \in \mA \ ;
\]
moreover, (\ref{eq_jers}) implies that, for every $\psi \in j(\wE)$,
\[
\psi t^* \varphi = 
\sigma (t^* \varphi) \psi = 
\rho^r (t)^* \sigma^r(\varphi) \psi =
t^* \eps (1,r) \eps (r,1) \varphi \psi =
t^* \varphi \psi
\ .
\]
We conclude that $t^* \varphi$ commutes with $\mA$ and $j(\wE)$, so that $t^* \varphi \in \mC_\mE$.
To prove that the set $\left\{ t^* \varphi \right\}$ is total in $\mC_\mE$, we proceed as in \cite[Lemma 5(1)]{DR88}, \cite[Proposition 7.1]{Vas02g}. As a first step, we assume that $\mE$ is trivial, in such a way that we have a strongly continuous action $\alpha : \sud \to {\bf aut}_{X^\rho} \mC_\mE$ such that $\mC_\mE^\alpha =$ $\zro$. By Fourier analysis, the set of elements of $\mC_\mE$ that transform like vectors in irreducible representations of $\sud$ is dense in $\mC_\mE$. Thus, we consider $n$-ples $\left\{ T_i \right\} \subset$ $\mC_\mE$ such that 
\[
\alpha_y (T_i) = \sum_j T_j u_{ji} (y) \ \ , \ \ y \in \sud \ ,
\]
where $u$ is some irreducible representation of $\sud$. Since $u$ is a subrepresentation of some tensor power of the defining representation of $\sud$, we find that there is $r \in \bN$ and constant orthonormal sections $\left\{ \varphi_i \right\}$ of $\mE^r \simeq$ $X^\rho \times \bC^{d^r}$ transforming like the $T_i$'s; moreover, we may regard the $\varphi_i$'s as elements of $j ( \iota , \mE^r )$ fulfilling the relations $\varphi_i^* \varphi_j = \delta_{ij} 1$. In this way, we find
\[
W^* := \sum_i T_i \varphi_i^* \in \mA \ \ .
\]
Multiplying on the right by $\varphi_j$, we conclude $T_j = W^* \varphi_i$; moreover, it is clear that $W \in$ $( \iota , \rho^r )$, and
\[
\rho (W) = 
\sum_i \sigma (\varphi_i) T_i^* = 
\sum_i \eps (r,1) \varphi_i T_i^* =
\eps (r,1) W
\ \ ,
\]
i.e., $W \in ( \iota , \rho^r )_\eps$. This proves Point 2 for trivial vector bundles. In the general case, the same argument used for the proof of Point 1 shows that $\left\{ t^* \varphi \right\}$ is total in $\mC_\mE$.
\end{proof}

The following Theorem generalizes several results, namely \cite[Theorem 4.1]{DR89A} for $\mZ = \bC 1$, \cite[Theorem 7.2]{Vas02g} for $\rho$ symmetric and $c_1(\rho) = 0$, and \cite[Theorem 4.13]{BL04} (for a single endomorphism) in the case in which each $(\rhors)_\eps$, $r,s \in \bN$, is free as a Banach $\zro$-bimodule.

\begin{thm}
\label{thm_dual}
Let $( \rho , \eps , \mR ) \in$ ${\bf end} \mA$ be a quasi-special endomorphism with class $d \oplus c_1(\rho)$, $d \in \bN$, $c_1(\rho) \in$ $H^2(X^\rho,\bZ)$. For every rank $d$ vector bundle $\mE \to X^\rho$ with first Chern class $c_1 (\rho)$, the following 
are equivalent:
\begin{enumerate}
\item there is a section $s : X^\rho \to \Omega_\mE$, $q \circ s = id_{X^\rho}$ (recall (\ref{def_ome_0}));
\item there is a compact group bundle $\mG \subseteq \mcSUE$, and a dual action $\nu : \cog \to \mA$ such that $\nu (\ers)_\mG =$ $(\rhors)_\eps$, $r,s \in \bN$, and $\mA' \cap (\mA \rtimes_\nu \wa \mG) =$ $\mZ$. Moreover, $\mA$ is the fixed-point algebra with respect to the gauge $\mG$-action on $\mA \rtimes_\nu \wa \mG$.
\end{enumerate}
\end{thm}

\begin{proof}
{\bf (1) $\Rightarrow$ (2)}: By Lemma \ref{lem_sue}, Lemma \ref{lem_cp_str}, and from the fact that there is $s \in S_{X^\rho} (\Omega_\mE)$, we find that the triple $( \mB_\mE , \mcSUE , \alpha )$ fulfilles the properties required for the proof of \cite[Theorem 6.1]{Vas02g}. Thus, there is a group bundle $\mG \subseteq \mcSUE$ and a dynamical system $( \mF , \mG , \beta )$, endowed with a $\zro$-epimorphism $\eta : \mB_\mE \to \mF$ which is injective on $\mA$, and such that $\eta (\mA) =$ $\mF^\beta$, $\eta (\mA)' \cap \mF =$ $\eta (\mZ)$. We recall that $\eta$ is defined as the quotient 
\begin{equation}
\label{def_eta}
\mB_\mE \ \to \ \mB_\mE \ / \ \left( C_0 ( \Omega_\mE - s(X^\rho) ) \mB_\mE \right)
\end{equation}
(note that $C_0 ( \Omega_\mE - s(X^\rho) )$ is an ideal of $\mC_\mE$), whilst
\begin{equation}
\label{def_g}
\mG :=
\left\{
y \in \mcSUE : \alpha^* ( s(x) , y ) = s (x)
\ , \ x := p(y)
\right\} \ .
\end{equation}
To economize in notations, we define
\[
\mA_\eta := \eta (\mA)
\ \ , \ \ 
\mZ_\eta := \eta (\mZ)
\ \ , \ \
\rho_\eta := \eta \circ \rho \circ \eta^{-1} |_{\eta(\mA)}  \in {\bf end} \mA_\eta
\ ,
\]
in such a way that we have a $\zro$-isomorphism $\eta |_\mA : (\mA , \rho) \to$ $(\mA_\eta , \rho_\eta)$. By definition of $\eta$, $\mG$, we obtain 
\begin{equation}
\label{eq_g_eta}
\wa \eta \circ \alpha (y,w) = \beta ( y , \wa \eta (w)  )
\ \ , \ \
y \in \mG \ , \ w \in \wa \mB_\mE
\end{equation}
(see \cite[\S 6.1]{Vas02g}). Now, since $\eta$ is a $\zro$-morphism, we find that $\eta$ restricts to a unitary isomorphism from $j(\wE)$ onto $\eta \circ j (\wE)$ (in fact if $\psi , \psi' \in \wE$ and $f :=$ $\left \langle \psi , \psi' \right \rangle$, then $f =$ $\eta( j(\psi)^* j(\psi') ) =$ $j(\psi)^* j(\psi')$). This fact has two consequences: first, by universality of the Cuntz-Pimsner algebra there is a $C(X)$-monomorphism $j_\eta : \coe \hra$ $\mF$, $j_\eta :=$ $\eta \circ j$; moreover, since $\zro$ is contained in the centre of $\mF$, if $\left\{ \psi_l \right\}$ is a finite set of generators for $j_\eta(\wE)$ then we can define
$\sigma_\eta \in$${\bf end}_{X^\rho} \mF$,
$\sigma_\eta (t) :=$ $\sum_l \psi_l t \psi_l^*$,
in such a way that $\sigma_\eta \circ \eta =$ $\eta \circ \sigma$, $\sigma_\eta \circ j_\eta =$ $j_\eta \circ \sigma_\mE$. This also implies $\sigma_\eta |_{\mA_\eta} = \rho_\eta$.
Applying (\ref{def_ada}) to $\mB_\mE$, and using (\ref{eq_g_eta}), we find that $j_\eta$ is $\mG$-equivariant, i.e.
\[
\beta ( y , \wa j_\eta (\xi) ) = \wa j_\eta \circ \wa y (\xi)
\ \ , \ \  
y \in \mG , \xi \in \wa \mO_\mE 
\ \ .
\]
Thus, $j_\eta$ restricts to a $\zro$-monomorphism $\nu : ( \mO_\mG , \sigma_\mG ) \to$ $( \mA_\eta , \rho_\eta )$. If $t \in (\sgrs)$, then 
\[
\nu (t) \rho_\eta^r (a) = 
\eta ( j(t) \rho^r(a') ) =
\eta ( j(t) \sigma^r(a') ) =
\eta ( \sigma^s (a') j(t) ) =
\rho_\eta^s (a) \nu (t) 
\ ,
\]
where $a' \in \mA$, $a := \eta (a') \in$ $\mA_\eta$; this implies $\nu ( \ers )_\mG \subseteq$ $( \rho_\eta^r , \rho_\eta^s )$, $r,s \in \bN$. Thus, $\nu$ is a dual action, and $\mF$ fulfilles the universal properties
\[
\left\{
\begin{array}{ll}
\nu (t) = j_s (t) \ \ , \ \ t \in \mO_\mG
\\
\psi a = \rho_\eta (a) \psi \ \ , \ \ a \in \mA_\eta , \psi \in j_s(\wE) \ .
\end{array}
\right.
\]
By identifying $\mA$ and $\mA_\eta$, we conclude that $\mF$ is isomorphic to $\mA \rtimes_\nu \wa \mG$. By construction of $\mF$, we also find $\mA' \cap ( \mA \rtimes_\nu \wa \mG ) =$ $\mZ$, and that $\mA$ is the fixed-point algebra with respect to the $\mG$-action on $\mA \rtimes_\nu \wa \mG$. Finally, the fact that $(\rhors)_\eps =$ $\nu (\ers)_\mG$, $r,s \in \bN$, follows from Lemma \ref{lem_da_dual}. 
{\bf (2) $\Rightarrow$ (1)}: applying Lemma \ref{lem_funct}, we find that $\nu$ restricts to the dual action $\mu : \mO_\mcSUE \to \mA$, $\mu (t) :=$ $\nu (t)$, $t \in \mO_\mcSUE \subset$ $\cog$, introduced in Theorem \ref{sue_action}, in fact $\nu (\theta) = \eps$, $\nu (\iota , \lambda \mE) =$ $\mR$. Again by Lemma \ref{lem_funct}, we find that there is a $\zro$-epimorphism $\eta : \mB_\mE \to \mA \rtimes_\nu \wa \mG$. Now, since $\mA' \cap (\mA \rtimes_\nu \wa \mG) = \mZ$, we find that the centre of $\mA \rtimes_\nu \wa \mG$ coincides with $\zro$, in fact $z \in \mZ$ belongs to the centre of $\mA \rtimes_\nu \wa \mG$ if and only if $z \psi =$$\psi z =$ $\rho (z) \psi$, $\psi \in j(\wE)$. Thus, $\eta (\mC_\mE) =$ $\zro$, and $\eta |_{\mC_\mE} : \mC_\mE \to \zro$ is a $\zro$-epimorphism. By Gel'fand equivalence, we conclude that there is a section $s : X^\rho \to \Omega_\mE$.
\end{proof}

\begin{defn}
\label{def_gauge}
The triple $(  \mA \rtimes_\nu \wa \mG , \mG , \beta )$ of Point 2 of Theorem \ref{thm_dual} is called a {\bf Hilbert extension of} $(\mA,\rho)$. In such a case, we say that the pair $( \mE , \mG )$ is a {\bf gauge-equivariant pair associated with} $( \mA , \rho )$; in particular, the group bundle $\mG \to X^\rho$ is called {\bf a gauge group of} $( \mA , \rho )$. 
\end{defn}

To emphasize the dependence of $\mG$ on $s \in S_X(\Omega_\mE)$ we use the notation $\mG \equiv \mG_s$. We now investigate the dependence of the system $( \mA \rtimes_\nu \wa \mG_s , \mG_s , \beta )$ on the section $s$. As a preliminary remark, we note that (\ref{def_aomega}) induces a group action
\begin{equation}
\label{def_act_sec}
\alpha_* : S_{X^\rho}(\Omega_\mE) \times \mSUE \to S_{X^\rho}(\Omega_\mE)
\ , \ 
\alpha_* (s,u) \ (x) :=  \alpha^* ( s(x) , u(x) )
\ \ , \ \
x \in X
\ \ .
\end{equation}

\begin{thm}
\label{thm_uni}
Let $\mG_s$, $\mG_{s'} \subseteq \mcSUE$ be compact group bundles with dual actions $\nu : ( \mO_{\mG_s} , \sigma_{\mG_s} ) \to$ $(\mA,\rho)$, $\nu' : ( \mO_{\mG_{s'}} , \sigma_{\mG_{s'}} ) \to$ $(\mA,\rho)$, such that $\mA' \cap (\mA \rtimes_\nu \wa \mG_s) = \mZ$, $\mA' \cap (\mA \rtimes_{\nu'} \wa \mG_{s'}) = \mZ$. Then the following are equivalent:
\begin{enumerate}
\item There is $u \in \mSUE$ with $s' = \alpha_* (s,u)$;
\item There is an isomorphism $\delta : \mA \rtimes_\nu \wa \mG_s \to$ $\mA \rtimes_{\nu'} \wa \mG_{s'}$ such that $\delta |_\mA = id_\mA$ (in the sense that $\delta$ is an $\mA$-module map), $\delta \circ \sigma_\eta =$ $\sigma_{\eta'} \circ \delta$.
\item $\mG_s$ and $\mG_{s'}$ are conjugates in $\mcSUE$, i.e. there is $u \in \mSUE$ such that $\mG_{s'} = u \mG_s u^*$.
\end{enumerate}
\end{thm}

\begin{proof}
{\bf (1) $\Rightarrow$ (2)}: We retain the notation introduced in the proof of Theorem \ref{thm_dual}. Since $s' = \alpha_* (s,u)$, we find $\alpha_u ( C_0 ( \Omega_\mE - s (X^\rho) ) ) =$ $C_0 ( \Omega_\mE - s'(X^\rho) )$. This implies that if we consider the $\zro$-epimorphisms $\eta : \mB_\mE \to \mA \rtimes_\nu \wa \mG_s$, $\eta' : \mB_\mE \to \mA \rtimes_{\nu'} \wa \mG_{s'}$ (see (\ref{def_eta})), then we obtain $\alpha_u( \ker \eta) =$ $\ker \eta'$. This last equality allows one to define
\[
\delta : \mA \rtimes_\nu \wa \mG_s \to \mA \rtimes_{\nu'} \wa \mG_{s'}
\ \ , \ \
\delta \circ \eta (b) := \eta' \circ \alpha_u (b) 
\ , \ 
b \in \mB_\mE
\ .
\]
Since $\eta$, $\eta'$ are faithful on $\mA$, we conclude that $\delta$ is faithful on $\mA$. For the same reason, $\delta$ is faithful on $j_s (\coe) \subset$ $\mA \rtimes_\nu \wa \mG_s$. Finally, since $\sigma_\eta \circ \eta =$ $\eta \circ \sigma$, $\sigma_{\eta'} \circ \eta' =$ $\eta' \circ \sigma$, we conclude that $\delta \circ \sigma_\eta = \delta \circ \sigma_{\eta'}$.
{\bf (2) $\Rightarrow$ (3)}: As a preliminary remark, we note that the minimality condition $\mA' \cap ( \mA \rtimes_\nu \wa \mG_s ) =$ $\mZ$ implies that the centre of $\mA \rtimes_\nu \wa \mG_s $ is $\zro$. By construction of $\sigma$, it is clear that $j_s (\wE) \subseteq$ $( \iota , \sigma )$ ; on the other side, if $\psi \in ( \iota , \sigma )$, then $c_l := \psi_l^* \psi \in ( \ii ) =$ $\zro$, and this implies that $\psi = \sum_l \psi_l c_l \in$ $j_s(\wE)$. Thus, $( \iota , \sigma ) = j_s(\wE)$; since $\delta (\iota , \sigma) =$ $( \iota , \sigma' )$, we conclude that $\delta$ defines a unitary $\zro$-module operator $u : \wE \to \wE$, which extends to a $\zro$-automorphism $\wa u \in {\bf aut}_{X^\rho} \coe$ fulfilling 
\begin{equation}
\label{eq_bj}
\delta \circ j_s = j_{s'} \circ \wa u
\ \ .
\end{equation}
Since $\delta (R) = R$, $R \in \mR \subset \mA$, and since $\mR =$ $j_s( \iota , \lambda \mE ) =$ $j_{s'} ( \iota , \lambda \mE )$, we conclude that $\wa u$ restricts to the identity on elements of $( \iota , \lambda \mE )$. In other terms, $u \in \mSUE$. Now, $j_s |_{\cog} = \nu$, $j_{s'} |_{\mO_{\mG_{s'}}} = \nu'$; thus, applying (\ref{eq_bj}), and by using the fact that
\[
\delta |_\mA = id_\mA 
\ \ \Rightarrow \ \
\delta \circ \nu = \nu
\ ,
\]
we find 
\[
\nu = \nu' \circ \wa u
\ \ .
\]
By Point 2 of Theorem \ref{thm_dual} $\nu : \mO_{\mG_s} \to \soro$ and $\nu' : \mO_{\mG_{s'}} \to \soro$ are isomorphisms, thus we have that $\wa u \in {\bf aut}_{X^\rho} \coe$ restricts to a $\zro$-isomorphism from $\mO_{\mG_s}$ to $\mO_{\mG_{s'}}$. From Lemma \ref{lem_str_og} we conclude that $\mG_s$ and $\mG_{s'}$ are conjugates.
{\bf (3) $\Rightarrow$ (1)}: By Lemma \ref{lem_str_og}, we have that $\wa u \in$ ${\bf aut}_{X^\rho} \coe$ restricts to an isomorphism from $\mO_{\mG_{s'}}$ onto $\mO_{\mG_s}$. Moreover, by (\ref{def_asue}) there is an automorphism $\alpha_u \in$ ${\bf aut}_{X^\rho} \mB_\mE$ such that $\alpha_u \circ j =$ $j \circ \wa u$, where $j$ is defined by (\ref{def_j}).
Let us now define
\[
\left\{
\begin{array}{ll}
j \ ' := j_s \circ \wa u \ : \ \coe \to \mA \rtimes_{\nu_s} \wa \mG_s
\\
\nu' := \nu \circ \wa u \ : \ \mO_{\mG_{s'}} \to \mA \rtimes_{\nu_s} \wa \mG_s \ .
\end{array}
\right.
\]
Since $j \ ' (t) = \nu' (t)$, $t \in \mO_{\mG_{s'}}$, and since $\psi a =$ $\rho (a) \psi$, $a \in \mA$, $\psi \in j \ '(\wE)$, we conclude that $\mA \rtimes_{\nu_s} \wa \mG_s$ fulfilles the universal properties (\ref{30}) for the crossed product $\mA \rtimes_{\nu_{s'}} \wa \mG_{s'}$. Thus, by universality there is an isomorphism
\[
\beta_u : \mA \rtimes_{\nu_{s'}} \wa \mG_{s'}
          \to
          \mA \rtimes_{\nu_s} \wa \mG_s 
\ \ , \ \
\beta_u \circ j_{s'} (t) = j \ ' (t)
\ ,
\]
$t \in \coe$. In other terms, we have $\beta_u \circ j_{s'} (t) = j_s \circ \wa u (t)$. In this way, we obtain a commutative diagram
\[
\xymatrix{
                    \mB_\mE
		    \ar[d]_-{\eta}
		    \ar[r]^-{\alpha_u}
		 &  \mB_\mE
		    \ar[d]^-{\eta'}
		 \\ \mA \rtimes_{\nu_{s'}} \mG_{s'}
		    \ar[r]^-{\beta_u}
		 &  \mA \rtimes_\nu \mG_s
}
\]
where $\eta, \eta'$ are defined as in (\ref{def_eta}). From the above diagram, it is evident that $\ker ( \eta \circ \alpha_u ) =$ $\ker ( \beta_u \circ \eta' ) =$ $\ker \eta'$. In particular, $c \in \mC_\mE \cap \ker ( \eta \circ \alpha_u )$ if and only if $c \in \mC_\mE \cap \ker \eta'$. But by construction of $\eta , \eta'$, we have 
\[
\left\{
\begin{array}{ll}
\mC_\mE \cap \ker ( \eta \circ \alpha_u ) = 
C_0 ( \Omega_\mE - \alpha_* ( s , u^* ) (X^\rho) )
\\
\mC_\mE \cap \ker \eta' = C_0 ( \Omega_\mE - s' (X^\rho) )
\end{array}
\right.
\]
From the above equalities, we conclude that $s'(X^\rho) =$ $\alpha_* ( s , u^* ) (X^\rho)$; so that $s' = \alpha_* ( s , u^* )$, and the Theorem is proved (up to a rescaling $u^* \mapsto u$).
\end{proof}

\begin{defn}
\label{def_he}
Let $( \rho , d , \mR )$ be a quasi-special endomorphism. Hilbert extensions $( \mA \rtimes_\nu \wa \mG_s , \beta , \mG_s )$, $( \mA \rtimes_\nu \wa \mG_s , \beta , \mG_{s'} )$ of $( \mA , \rho )$ are said to be {\bf equivalent} when Point 2 of the previous theorem is fulfilled.
\end{defn}

\begin{ex}
{\it
Let $\mE \to X$ be a vector bundle, $\mG \subseteq \mcSUE$ a group bundle. Then the canonical endomorphism $\sigma_\mG$ is special, and $\coe \simeq \cog \rtimes_\mu \wa \mG$, where $\mu$ is the identity of $\mO_\mG$.
}
\end{ex}

\begin{ex}
{\it
The superselection structures considered in \cite{BL97,BL04} define endomorphisms with permutation quasi-symmetry. In particular, the endomorphisms $\rho \in {\bf end} \mA$ considered in \cite[\S 7]{Vas02g} are quasi-special, and have trivial Chern class $c_1(\rho)$. Thus, we may pick a trivial vector bundle $\mE :=$ $X^\rho \times \bC^d$, and construct the crossed product $\mB_\mE$ by the dual action of $X^\rho \times \sud$. It is proved in \cite[Lemma 7.3]{Vas02g} that the centre of $\mB_\mE$ is isomorphic to $\zro \otimes C(G \backslash \sud)$, where $G \subseteq \sud$ is a compact group unique up to conjugation in $\sud$; thus, it is clear that the spectrum $\Omega_\mE =$ $X^\rho \times G \backslash \sud$ admits sections, and we can construct the crossed product $\mA \rtimes_\nu \wa \mG$, where $\mG =$ $X^\rho \times G$. 
Now, we may consider as well a rank $d$ vector bundle $\mE' \to X^\rho$ with trivial first Chern class, and construct the crossed product $\mB_{\mE'}$ by the dual $\mcSUE'$-action: as shown in \cite[\S 6.2]{Vas02g}, it is not ensured that the resulting gauge group $\mG' \to X^\rho$ is isomorphic to $\mG$.
We conclude that unicity of the gauge group of $( \mA , \rho )$ is ensured when we restrict ourselves to consider trivial group bundles acting on trivial vector bundles. In this way, we obtain the "algebraic Hilbert spaces" considered in the above-cited references.
}
\end{ex}

\subsection{The moduli space of Hilbert extensions.}

Lacking of existence and unicity of the Hilbert extension has been already discussed in the case in which the intertwiners spaces are locally trivial as continuous fields of Banach spaces by means of cohomological methods (\cite{Vas06p}). 
In the present section we present some immediate consequences of Theorem \ref{thm_uni}: this will allows us to give a classification of the Hilbert extensions of a quasi-special \sC dynamical system by means of the space of sections of a bundle, without any assumption on the structure of the intertwiner spaces.

Let $( \rho , d , \mR ) \in$ ${\bf end} \mA$ be a weakly special endomorphism. To simplify the exposition we assume that $X^\rho$ is connected (otherwise, it is possible to decompose $( \mA , \rho )$ into a direct sum indexed by the clopens of $X^\rho$). We denote the set of isomorphism classes of vector bundles $\mE \to X^\rho$ with rank $d$ and first Chern class $c_1(\rho)$ by
$\mcE (\rho)$,
and define the Abelian \sC algebra
\[
\mC_\rho := \bigoplus_{[\mE] \in \mcE (\rho)} \mC_\mE \ .
\]
In the previous definition we considered the \sC algebras $\mC_\mE$, thus for each $[\mE] \in \mcE (\rho)$ we made a choice of $\mE$ in its class $[\mE] \in \mcE (\rho)$. If $\mE$ and $\mE'$ are isomorphic (i.e., $[\mE] = [\mE'] \in \mcE (\rho)$), then $\mC_\mE$ is isomorphic to $\mC_{\mE'}$, thus the isomorphism class of $\mC_\rho$ does not depend on the choice.

Let us now introduce the compact group bundle:
\[
\mathcal{SUE} (\rho) \ := \ \prod_{\mE \in \mcE (\rho)} \mcSUE \ \to \ X^\rho \ ;
\] 
since each component $\mcSUE$, $\mE \in \mcE (\rho)$, of $\mathcal{SUE} (\rho)$ is full, we conclude that the set of sections ${\bf{SU}}\mcE (\rho)$ is a section group for $\mathcal{SUE} (\rho)$.

\begin{lem}
$\mC_\rho$ is a continuous bundle of Abelian \sC algebras with base space $X^\rho$. Moreover, there is a gauge action
$\alpha :$ $\mathcal{SUE} (\rho) \times_{X^\rho} \wa \mC_\rho \to$ $\wa \mC_\rho$.
\end{lem}

\begin{proof}
Since each $\mC_\mE$, $E \in \mcE (\rho)$, is a continuous bundle of Abelian \sC algebras, we find that $\mC_\rho$ is a continuous bundle of Abelian \sC algebras, with associated \sC bundle $\wa \mC_\rho \to X^\rho$. Note that each fibre $\mC_{\rho,x}$, $x \in X^\rho$, is isomorphic to the direct sum $\oplus_E \mC_{\mE,x}$ of the fibres of the \sC algebras $\mC_\mE$, $\mE \in \mcE(\rho)$. Let us denote the generic element of $\mathcal{SUE} (\rho)$ by $\widetilde u :=$ $\left\{ u_E \in \mcSUE \right\}_{ \mE \in \mcE (\rho) }$; for every $c :=$ $\oplus_E c_E \in$ $\wa \mC_\rho$ (with $c_E \in \wa \mC_\mE$), we define 
\[
\alpha ( \widetilde u , c ) := \oplus_E \ \alpha_E ( u_E , c_E )
\ ,
\]
where $\alpha^E : \mcSUE \times_{X^\rho} \wa \mC_\mE \to \wa \mC_\mE$ is the gauge action defined as in (\ref{def_asue}), restricted to $\mC_\mE$. 
\end{proof}

It follows from the previous Lemma that the spectrum $\Omega_\rho$ of $\mC$ defines a bundle $q : \Omega_\rho \to X^\rho$. By definition of $\mC_\rho$, we may regard $\Omega_\rho$ as the disjoint union $\dot{\cup}_\mE \Omega_\mE$. For every $x \in X^\rho$, the fibre $\Omega_{\rho,x} := q^{-1}(x)$ is homeomorphic to the disjoint union $\dot{\cup}_{|\mcE (\rho)| \ } G_x \backslash \sud$, where $G_x \subseteq \sud$ is the group defined by means of Theorem \ref{sue_action} (see also Remark \ref{rem_c}). 
By Gel'fand duality, there is a gauge action
\begin{equation}
\label{def_act_omro}
\alpha^* : \mathcal{SUE} (\rho) \times_{X^\rho} \Omega_\rho \to \Omega_\rho \ \ .
\end{equation}
Let us consider the space of sections $S_{X^\rho}(\Omega_\rho)$ of $\Omega^\rho$. Since $X^\rho$ is connected, every $s \in S_{X^\rho} (\Omega_\rho)$ has image contained in some connected component of $\Omega_\rho$. Since each $\Omega_\mE$ appears as a clopen in $\Omega_\rho$, we conclude that $s (X^\rho) \subseteq \Omega_\mE$ for some $\mE \in$ $\mcE(\rho)$, so that $s$ is actually a section of $\Omega_\mE$.
Now, the action (\ref{def_act_omro}) induces an action
\[
\alpha_* : {\bf{SU}}\mcE (\rho) \times S_{X^\rho}(\Omega_\rho) \to S_{X^\rho}(\Omega_\rho)
\ \ , \ \
\alpha_*(u,s) (x) := \alpha^* ( u(x) , s(x) ) 
\ \ , \ \ 
x \in X \ .
\]
The previous elementary remarks, and Theorem \ref{thm_uni}, imply the following
\begin{thm}
\label{thm_moduli}
Let $( \rho , d , \mR )$ be a quasi-special endomorphism of a \sC algebra $\mA$. (If $X^\rho$ is connected,) then there is a one-to-one correspondence between the set of Hilbert extensions of $( \mA , \rho )$ and the space of sections $S_{X^\rho} (\Omega_\rho)$, assigning to each $s \in S_{X^\rho} (\Omega_\rho)$ the triple $( \mA \rtimes_\nu \wa \mG_s , \mG_s , \beta )$. Two Hilbert extensions of $( \mA , \rho )$ associated with sections $s,s'$ are equivalent if and only if $s' = \alpha_* ( s , u )$ for some $u \in {\bf{SU}}\mcE (\rho)$.
\end{thm}

Some remarks follow. 

First, the fibration $\Omega_\rho \to X^\rho$ may lack sections, and in this case there are no Hilbert extensions of $(\mA,\rho)$. Moreover, the ${\bf{SU}}\mcE (\rho)$-action on $S_{X^\rho} (\Omega_\rho)$ may be not transitive, and this means that there could be non-equivalent Hilbert extensions of $(\mA,\rho)$. It could be interesting to study the behaviour of Hilbert extensions $\mA \rtimes_{\nu} \wa \mG_{s_n}$ for sequences $\left\{ s_n \right\}$ converging to a given $s \in S_{X^\rho} ( \Omega_\rho )$.

The case with $\mZ = \bC 1$ studied in \cite[\S 4]{DR89A} yields a special endomorphism $\rho$ with $\zro \simeq \bC$, so that $\mcE (\rho)$ has as unique element the Hilbert space $\bC^d$, $d \in \bN$, and $\Omega_\rho$ reduces to a homogeneous space; this means that there is a unique (up to equivalence) Hilbert extension. More generally, if $\mZ \neq \bC 1$, a canonical endomorphism $\rho$ in the sense of \cite[\S 4]{BL04} which is also quasi-special has Chern class $c_1(\rho) = 0$, and a Hilbert extension of $(\mA,\rho)$ in the sense of the above-cited reference corresponds to a constant section of $\Omega_\mE \simeq$ $X^\rho \times G \backslash \sud$, $\Omega_\mE \subset \Omega_\rho$, where $\mE \in \mcE (\rho)$ is the trivial rank $d$ vector bundle over $X^\rho$. 
%
%By considering non-constant sections, we could get Hilbert extensions not equivalent 
%to the ones considered in \cite{BL04}.
%
%
%In the next subsections, we give some classes of examples.
%
%
\subsubsection{A class of examples. Non-uniqueness of the Hilbert extension.} \label{ex_non_uni} Let $\mA := \mO_{\mcSUE}$, $\rho :=$ $\sigma_\mcSUE$ be defined as in \S \ref{key_group_duals} for a fixed rank $d$ vector bundle $\mE \to X$, $d \in \bN$. There is a natural identification $X \equiv X^\rho$, and $\rho$ is a special endomorphism with class $d \oplus c_1(\mE)$ (see \S \ref{sec_spe}). 
Now, if $\mE' \to X$ is an arbitrary vector bundle having the same rank and first Chern class as $\mE$, then the dual action
\begin{equation}
\label{eq_da_ex}
\mu' : \mO_{\mcSUE'} \to \mA
\end{equation}
defined as in Theorem \ref{sue_action} is an isomorphism; in fact, $\mA$ is the $\rho$-stable algebra generated by $\theta$, $( \iota , \lambda \mE )$ (see \cite[Proposition 4.17]{Vas04}), and this is exactly the image of $\mu'$ (see remarks after Theorem \ref{sue_action}). In this way, by \cite[Example 3.2]{Vas05} we find
\[
\mO_{\mE'} 
\simeq  
\mA \rtimes_{\mu'} \wa{\mcSUE'}
\]
(in particular, $\coe \simeq$ $\mA \rtimes_\mu \wa{\mcSUE}$, see also Lemma \ref{lem_ex_mg} below). Since $\mC_{\mE'} :=$ $\mO_{\mE'}' \cap \mO_{\mE'} =$ $C(X)$, we conclude that $\Omega_{\mE'} = X$, so that $\Omega_\rho$ is a disjoint union of copies of $X$:
\[
\Omega_\rho 
\ \simeq \ 
\dot{\cup}_{\mE' \in \mcE (\rho)} X
\ .
\]
This means that for every $\mE' \in \mcE (\rho)$ there is a unique section $s_{\mE'} : X \to \Omega_\rho$, with image coinciding with the copy of $X$ labelled by $\mE'$.
For each of such sections, there is a Hilbert extension $( \mO_{\mE'} , \mcSUE' )$ of $( \mA , \rho )$, not necessarily equivalent to $( \coe , \mcSUE) $. For example, take $X$ coinciding with the sphere $S^{2n}$, $n > 2$, and $\mE$ such that $\mcSUE$ is nontrivial (it is well-known that such vector bundles exist on $S^{2n}$, see \cite[I.3.13]{Kar}). Since $c_1 (\mE) =$ $c_1 (\rho) =$ $0$ (in fact, $H^2 ( S^{2n} , \bZ ) = 0$), we may pick $\mE' :=$ $S^{2n} \times \bC^d$, the associated dual action (\ref{eq_da_ex}), and obtain the "trivial" Hilbert extension $( \ C(S^{2n}) \otimes \mO_d \ , \ S^{2n} \times \sud \ )$, with $\mcSUE$ not isomorphic to $S^{2n} \times \sud$.

\subsubsection{Another class of examples. Non-existence of the Hilbert extension.} \label{ex_non_exi} Let $X$ denote the $2$-sphere. We consider a locally trivial principal $\bSO(3)$-bundle $\Omega \to X$ {\em with no sections}, and endowed with the translation action $\lambda : \bSO(3) \to {\bf aut}_X C(\Omega)$ such that $C(\Omega)^\lambda = C(X)$. 
Such a bundle exists because $H^1(X,\bSO(3)) \simeq$ $\pi_1(\bSO(3)) \simeq$ $\bZ_2$
is non trivial.
Now, $\bSO(3)$ is a quotient of $\mathbb{SU}(2)$ (in fact, $\bSO(3) \simeq$ $\mathbb{SU}(2) / \bZ_2$), thus we may lift $\lambda$ to an action $\lambda : \mathbb{SU}(2) \to {\bf aut}_X C(\Omega)$. 
%
%By \cite[Lemma 5.1]{Vas02g}, we find that each $\Omega_x$, $x \in X$, 
%is a homogeneous $\mathbb{SU}(2)$-space, and each (transitive) action 
%$\lambda_x : \mathbb{SU}(2) \to C(\Omega_x)$, $x \in X$, may be regarded as the 
%right translation action. 
%
By Gel'fand duality, we write $\lambda_u c (\omega) =$ $c (\omega u^{-1})$, $u \in \mathbb{SU}(2)$, $c \in C(\Omega)$, $\omega \in \Omega$.
%
%Moreover, by local triviality of $\Omega$ there is $G \subseteq \mathbb{SU}(2)$ 
%with homeomorphisms $\pi_U : \Omega |_U \to U \times G \backslash \mathbb{SU}(2)$, 
%for suitable open sets $U \subseteq X$ covering $X$.
%
We now consider the \sC algebra $\mB := C( \Omega , \mO_2 )$, endowed with the action
\[
\alpha : \mathbb{SU}(2) \to {\bf aut}_X \mB
\ \ , \ \
\alpha_u b (\omega) := \wa u (b(\omega u^{-1}))
\ ,
\]
$b \in \mB$, $u \in \mathbb{SU}(2)$, $\omega \in \Omega$, defined as in (\ref{eq_act_od}). As in \cite[Lemma 6.8]{Vas02g}, we find $\mA' \cap \mB =$ $\mB' \cap \mB =$ $C(\Omega)$, where $\mA :=$ $\mB^\alpha$ the fixed-point $C(X)$-algebra with respect to the $\mathbb{SU}(2)$-action.

We now equip $\mB$ with the endomorphism $\sigma (b) :=$ $\sum_{i=1}^2 \psi_i b \psi_i^*$, $b \in \bN$, where $\psi_1 , \psi_2$ are the isometries generating $\mO_2 \subset$ $\mB$. By a standard argument, we find that $\sigma$ restricts to an endomorphism $\rho \in {\bf end}_X \mA$ (see (\ref{eq_rest}) below). 
Let us denote the trivial rank $2$ vector bundle by $\mE \to X^\rho$, and its special unitary bundle by $\mcSUE =$ $X \times \mathbb{SU}(2)$. Then $\coe =$ $C ( X , \mO_2 )$ and $\mO_{\mcSUE} =$ $C(X , \mO_{\mathbb{SU}(2)} )$. If we regard $C(X)$ as the \sC subalgebra of $C(\Omega)$ of functions which are invariant with respect to the right translation action induced by $\lambda$, then we find an obvious inclusion $\mO_\mcSUE \subseteq \mA$; moreover, it is clear that $\coe \subseteq \mB$. If we denote the above-mentioned inclusions by $\mu : \mO_{\mcSUE} \to \mA$ and $j : \coe \to \mB$, then it is clear that $j(t) = \mu (t)$, $t \in \mO_\mcSUE$. Thus, by universality of the crossed product by a dual action, we conclude the following:
\begin{prop}
$\rho \in {\bf end}_X \mA$ is a special endomorphism with class $2 \oplus 0 \in$ $\bN \oplus H^2(X,\bZ)$, and $\mB$ is canonically isomorphic to the crossed product $\mA \rtimes_\mu \wa \mcSUE$.
\end{prop}
Thus, $\Omega$ coincides with the bundle $\Omega_\mE$ defined in the previous sections, and $\Omega_\mE$ lacks of sections. Moreover, $X$ being the two sphere, every vector bundle $\mE' \to X$ with rank $2$ and trivial first Chern class is isomorphic to $\mE$ (see \cite[V.3.25]{Kar}). Thus, we conclude that $\Omega_\rho = \Omega_\mE$ lacks of sections, and $(\mA,\rho)$ does not admit Hilbert extensions.

\subsection{The model, and a duality theorem.}
\label{duality}

Let $X$ be a compact Hausdorff space $X$, and $\mZ$ an Abelian $C(X)$-algebra with identity $1$. We consider a $C(X)$-Hilbert $\mZ$-bimodule $\mM \simeq \wE \otimes_{C(X)} \mZ$ with left action $\phi :$ $\mZ \to$ $( \mM , \mM )$, defined as in Example \ref{ex_mez} for some rank $d$ vector bundle $\mE \to X$. The Cuntz-Pimsner algebra $\com$ associated with $\mM$ may be presented as the one generated by $\mZ$ and a set of generators $\left\{ \psi_l \right\}_{l=1}^n$ of $\wE$, with relations
\begin{equation}
\label{def_com}
\psi_l^* \psi_m = \left \langle \psi_l , \psi_m \right \rangle 
\ \ , \ \
\sum_l \psi_l \psi_l^* = 1
\ \ , \ \
z \psi_l = \phi (z) \psi_l
\ \ ,
\end{equation}
$z \in \mZ$, $\left \langle \psi_l , \psi_m \right \rangle \in$ $C(X)$. It is clear that $\com$ is a $C(X)$-algebra; on the other side, note that in general $\mZ$ is not contained in the centre of $\com$. It is easy to verify that the \sC bundle $\wa \mO_\mM \to X$ has fibres isomorphic to the Cuntz-Pimsner algebras $\mO_{\mM_x}$ associated with the fibres $\mM_x$, $x \in X$; we denote the fibre epimorphisms by $\eta_x : \com \to \mO_{\mM_x}$. 
Since each $\mM_x$ is isomorphic to $\bC^d \otimes \mZ_x$ as a right Hilbert $\mZ_x$-module (see Example \ref{ex_mez}), 
we find that there is a set $\left\{ \psi_{x,i} \right\}_{i=1}^d$ of orthonormal generators for $\mM_x$. At the level of 
Cuntz-Pimsner algebra, $\left\{ \psi_{x,i} \right\}_{i=1}^d$ appears as a set of orthonormal partial isometries with total 
support the identity $1_x \in \mO_{\mM_x}$; by universality of the Cuntz relations, we obtain a unital monomorphism 
$j_x : \mO_d \to \mO_{\mM_x}$.
Now, the following endomorphism is defined:
\begin{equation}
\label{def_sigma}
\tau_\mM (t) := \sum_l \psi_l t \psi_l^*
\ \ , \ \ 
t \in \com
\ .
\end{equation}
\noindent It turns out that $\tau_\mM$ is weakly-symmetric: the representation $\theta : \bP_\infty \to$ $\com$ is defined by 
\begin{equation}
\label{eq_pssm}
\theta (p) := 
\sum \psi_{l_{p(1)}} \cdots \psi_{l_{p(r)}} \psi_{l_r}^* \cdots \psi_{l_1}^*
\ \ , \ \
p \in \bP_r \subset \bP_\infty \ .
\end{equation}
Moreover, for every $x \in X$ we have a unital endomorphism
\[
\tau_x \in {\bf end} \mO_{\mM_x}
\ \ : \ \ 
\tau_x (t) := \sum_i \psi_{x,i} \ t \ \psi_{x,i}^*
\ ,
\]
in such a way that $\tau_x \circ \eta_x = \eta_x \circ \tau_\mM$, $x \in X$. By adapting a standard argument used in the setting of the Cuntz algebra, we find that if $u \in \ud$ and $u \psi_{x,i} :=$ $\sum_k \psi_{x,k} u_{ki}$ 
(where $u_{ki} \in \bC$ denotes the matrix coefficient of $u$), then
\begin{equation}
\label{rel_taux}
\tau_x (t) = 
\sum_i (u \psi_{x,i}) \ t \ (u \psi_{x,i})^*
\ .
\end{equation}
In the next Lemma, we prove that $\tau_\mM$ is special. To this end, for every $r \in \bN$ we consider the totally antisymmetric bundle $\lambda (\mE^r) \subseteq$ $\mE^{r d^r}$, and introduce the \sC subalgebra $\mS$ of $\com$ generated by elements of the spaces $( \iota , \lambda (\mE^r) ) \subseteq$ $( \iota , \tau_\mM^{rd^r} )_\theta$, $r \in \bN$.

\begin{lem}
\label{lem_om}
With the above notation, it turns out $\mS' \cap \com = \mZ$; moreover, $\tau_\mM$ is a special endomorphism, with $(\smrs) =$ $(\smrs)_\theta =$ $(\ers)$, $r,s \in \bN$.
\end{lem}

\begin{proof}
As a first step, note that we may identify $C(X)$ with the \sC algebra
\begin{equation}
\label{def_zz}
\left\{ z \in \mZ : \psi z = \phi(z) \psi , \psi \in \mM \right\}
\end{equation}
(otherwise, we may consider the spectrum $X'$ of (\ref{def_zz}), and pass to the pullback $\mE' :=$ $\mE \times_X X'$). Let $t \in$ $\mS' \cap \com$; then, for every $x \in X$ we find $\eta_x (t) \in$ $\eta_x (\mS)' \cap \mO_{\mM_x}$. Now, since each $\lambda (\mE^r) \to X$, $r \in \bN$, is a locally trivial line bundle, we find that $\eta_x (\mS)$ is generated by a sequence of partial isometries $\left\{ S_{x,r} \right\}_r$. Since $S_{x,r}$ is a generator of the totally antisymmetric space of $\mE^r_x$, we have the following analogue of (\ref{scp2}):
\[
S_{x,r}^* \tau_x^{rd^r} (S_{x,r}) = (-1)^{rd^r - 1} d^{-r} 1_x
\ .
\]
Thus, applying \cite[Proposition 3.5(b)]{DPZ97}, we conclude $\eta_x (t) \in \mZ_x$, and $t \in \mZ$.
Finally, by \cite[Lemma 5.5]{Vas04} we find $(\smrs) =$ $(\smrs)_\theta =$ $(\ers)$, so that $\tau_\mM$ has permutation symmetry; moreover, $\tau_\mM$ is also special, in fact $( \iota , \lambda \mE ) \subseteq$ $( \iota , \tau_\mM^d )$ fulfilles the required properties.
\end{proof}

Let now $p : \mG \to X$, $\mG \subseteq \mcUE$, be a compact group bundle; then, for every $x \in X$, the inclusion $G_x \subseteq \ud$ induces an action $G_x \to U \mM_x \simeq$ by right $\mZ_x$-module operators. Let $\wa \phi$ be the morphism defined in (\ref{eq_la_b}); we make the following assumptions: 
\begin{equation}
\label{eq_aez}
\left\{
\begin{array}{ll}
[ y , \wa \phi (v) ] = 0
\ \ , \ \ 
y \in \mG \ , \ v \in \mZ_{p(y)}
\\
(\mrs)_\mG \subseteq (\mrs)_\mZ
\ ,
\end{array}
\right.
\end{equation}
where $(\mrs)_\mZ$ is defined as in (\ref{def_bm}). The equality (\ref{eq_aez}.1) ensures that $\mM$ is a $\mG$-Hilbert $\mZ$-bimodule fulfilling (\ref{eq_tr_fa}), thus by Lemma \ref{lem_fa_om} there is a gauge action $\alpha : \mG \times_X \wa \mO_\mM \to$ $\wa \mO_\mM$. The equality (\ref{eq_aez}.2) ensures that $\mM_\mG^\otimes$ is a {\em tensor} \sC category, in accord with the remarks in Example \ref{ex_mez}.
Let us now consider the $C(X)$-algebra $\com^\mG \subseteq \com$ generated by the bimodules $(\mrs)_\mG$, $r,s \in \bN$. In the same way as in Lemma \ref{lem_og}, we find that if there is an invariant $C(X)$-functional defined on $\mG$, then $\com^\mG$ is the fixed-point algebra of $\com$ with respect to $\alpha$. 
If $y \in \mG_x$, $x \in X$, then 
\[
\alpha ( y , \psi_{x,i} ) = \sum_k \psi_{x,k} y_{ki} \ ;
\]
thus, by using (\ref{rel_taux}),
\begin{equation}
\label{eq_rest}
\alpha ( y , \tau_x (t) ) =
\sum_i (y \psi_{x,i}) \ \alpha ( y , t ) \ (y \psi_{x,i})^* =
\tau_x \circ \alpha ( y , t )
\ .
\end{equation}
The previous equality implies that (\ref{def_sigma}) restricts to an endomorphism $\tau_\mG \in {\bf end}_X \com^\mG$. Since $\theta (r,s) \in$ $( \smrs ) \cap \com^\mG$, we conclude that $\tau_\mG$ has weak permutation symmetry.

\begin{lem}
\label{lem_ex_mg}
Let $\mG \subseteq \mcSUE$ be a compact group bundle fulfilling (\ref{eq_aez}) and endowed with an invariant $C(X)$-functional $\varphi : C(\mG) \to C(X)$. Then the canonical endomorphism $\tau_\mG \in$ ${\bf end}_X \com^\mG$ is quasi-special, and $\com \simeq$ $\com^\mG \rtimes_\nu \wa \mG$ is a Hilbert extension of $( \com^\mG , \tau_\mG )$.
\end{lem}

\begin{proof}
It is clear that $\tau_\mG$ is weakly special, and that $\tau_\mG^s (\mZ) ( \taurs )_\theta \subseteq$ $( \taurs )$, $r,s \in \bN$. To prove the opposite inclusion we consider a finite set $\left\{ \psi_l \right\}$ of generators for $\wE$; by construction, $\psi_l \in$ $( \iota , \tau_\mM )_\theta$ for every index $l$, so that $\psi_M \in$ $( \iota , \tau_\mM^r )_\theta$, $|M| = r$. For every $t \in ( \taurs )$, we define $z_{LM} :=$ $\psi_L^* t \psi_M$, and note that for every $b \in \com^\mG$ it turns out
\[
b z_{LM} =
b \psi_L^* t \psi_M =
\psi_L^* \tau_\mG^s (b) t \psi_M =
\psi_L^* t \tau_\mG^r (b) \psi_M =
z_{LM} b \ .
\]
In particular, $z_{LM}$ commutes with elements of $( \iota , (\lambda \mE)^r ) \subseteq$ $( \iota , \mM^{rd^r} )_\mG$, $r \in \bN$; by Lemma \ref{lem_om}, we conclude that $z_{LM} \in$ $(\com^\mG)' \cap \com =$ $\mZ$. In this way, we find $t =$ $\sum_{LM} \psi_L z_{LM} \psi_M^* =$ $\sum_{LM} \rho^s(z_{LM}) \psi_L \psi_M^*$. Note that $\psi_L \psi_M^* \in$ $( \smrs )_\theta$, thus $t \in \rho^s(\mZ) ( \smrs )_\theta$. Now, for every $r,s \in \bN$, the invariant functional $\varphi$ induces an invariant mean
\[
m: ( \smrs )_\theta \to (\taurs)_\theta \ .
\]
Since $m$ is an $\com^\mG$-module map, we obtain
\[
t = m(t) = 
\sum_{LM} m \left( \rho^s(z_{LM}) \psi_L \psi_M^* \right) =
\sum_{LM} \rho^s(z_{LM}) \ m(\psi_L \psi_M^*)
\ ,
\]
with $m(\psi_L \psi_M^*) \in$ $( \taurs )_\theta$. This proves that $\tau_\mG$ is quasi-special.
We now proceed with the proof that $\com$ is a Hilbert extension of $( \com^\mG , \tau_\mG )$. To this end, we note that by considering the inclusion $\com^\mG \subseteq \com$, it turns out $( \taurs )_\theta \subseteq$ $( \smrs )_\theta =$ $( \ers )$ (see Lemma \ref{lem_om}); in particular, it is clear by definition of $\alpha$ that $( \taurs )_\theta =$ $(\ers)_\mG$, $r,s \in \bN$. In this way, the inclusions $\cog \subseteq \com^\mG \subseteq \com$ and (\ref{def_sigma}) imply that $\com =$ $\com^\mG \rtimes_\nu \wa \mG$, where $\nu$ is the inclusion $\cog \subseteq \com^\mG$. Finally, since the \sC algebra $\mS$ is contained in $\com^\mG$ (in fact, $\mG \subseteq \mcSUE$, and every element of $\mS$ is $\mcSUE$-invariant), we conclude from Lemma \ref{lem_om} that $(\com^\mG)' \cap \com =$ $\mZ$.
\end{proof}

\begin{ex}
\label{ex_cg}
{\it
Let $\mZ$ be a unital $C(X)$-algebra, and $K$ a compact group. Then there is a discrete Abelian group $\EuFrak{C}(K)$ whose elements are classes $[v]$ of irreducible representations of $K$ with respect to a suitable equivalence relation (see \cite{BL04} and related references: $\EuFrak{C}(K)$ is called the chain group of $K$, and is isomorphic to the Pontryagin dual of the centre of $K$).
Let us now assume that there is an action $\alpha : \EuFrak{C}(K) \to {\bf aut}_X \mZ$. 
If $w$ is an irreducible representation of $K$ with rank $d$, then we consider the $C(X)$-Hilbert $\mZ$-bimodule $\mM_w :=$ $\bC^d \otimes \mZ$ with left action $z \psi :=$ $\psi \alpha_{[w]}(z)$, $\psi \in \mM_w$, $z \in \mZ$, and define $\mM :=$ $\mM_w \oplus \mM_w \oplus \mZ$ (where $\mZ$ is regarded as the free, rank one Hilbert $\mZ$-bimodule). In this way, $\mM$ becomes a $K$-Hilbert $\mZ$-bimodule, with action
\[
U : K \times \mM \to \mM
\ \ , \ \
U_y := w_y \oplus w_y \oplus 1
\ .
\]
It can be proved that $( \mrs )_K \subseteq$ $( \mrs )_\mZ$, $r,s \in \bN$, so that $\mM_K^\otimes$ is a tensor \sC category, and $\tau_K \in$ ${\bf end}_X \com^K$ is an endomorphism with permutation quasi-symmetry fulfilling $( \tau_K^r , \tau_K^s ) =$ $( \mrs )_K$, $r,s \in \bN$; moreover, $(\com^K)' \cap \com = \mZ$ (see \cite[\S 4]{LV07} for details). In particular, if $w$ takes values in $\sud$, then $\tau_K$ is quasi-special.
}
\end{ex}

We can now characterize quasi-special endomorphisms in terms of \sC dynamical systems of the type $( \com^\mG , \tau_\mG )$. The proof is a direct consequence of Theorem \ref{thm_dual0}, Eq.(\ref{def_sigmab}), Lemma \ref{lem_da_dual}, and Eq.(\ref{def_sigma}), so that it is omitted.
\begin{thm}
\label{car_rho}
Let $( \rho , \eps , \mR ) \in$ ${\bf end} \mA$ be a quasi-special endomorphism with class $d  \oplus c_1(\rho) \in$ $\bN \oplus H^2(X^\rho,\bZ)$. Then for every gauge-equivariant pair $( \mE , \mG )$ associated with $(\mA,\rho)$ (see Definition \ref{def_gauge}), there is a $\mG$-Hilbert $\mZ$-bimodule $\mM \simeq$ $\wE \otimes_{\zro} \mZ$ such that the following diagram is commutative:
\begin{equation}
\label{cd_1}
\xymatrix{
    ( \soro , \rho )
    \ar[r]^-{\subseteq}
    \ar[d]_-{\simeq}
 &  (\oro , \rho)
    \ar[d]^-{\simeq}
 \\ (\cog , \sigma_\mG)
    \ar[r]^-{\subseteq}
 &  (\cmg , \tau_\mG)
}
\end{equation}
\noindent (the symbol "$\simeq$" stands for an isomorphism) Moreover, there are isomorphisms of strict tensor \sC categories $\wa \mG \simeq \wa \rho_\eps$, $\mM_\mG^\otimes \simeq \wa \rho$ such that, for each $r,s \in \bN$, the following diagram is commutative:
\begin{equation}
\label{cd_2}
\xymatrix{
   {(\rhors)_\eps}
   \ar[d]_-{\simeq}
   \ar[r]^-{\subseteq}
&  (\rhors)
   \ar[d]^-{\simeq}
\\ (\ers)_\mG
   \ar[r]^-{\subseteq}
&  (\mrs)_\mG
}
\end{equation}
\end{thm}
%
%
%\begin{proof}
%We refer the reader to Theorem \ref{thm_dual0}, from which the isomorphisms $( \rhors ) \simeq$ $( \mrs )_\mG$, $r,s \in \bN$, follow. 
%In particular if $t \in (\rhors)_\eps$, then by using the usual decomposition $t =$ $\sum_{LM} \psi_L t_{LM} \psi_M^*$, $t_{LM} :=$
%$\psi_L^* t \psi_M$ (see (\ref{def_psil})), we find $t_{LM} \in \mZ$, but also
%%
%\[
%\rho (t_{LM})
%=
%\psi_L^* \eps (1,s) \eps (s,1) t \eps (1,r) \eps (r,1) \psi_M
%=
%t_{LM}
%\ .
%\]
%%
%This implies $t_{LM} \in$ $\zro$, so that $t \in (\ers)_\mG$. On the converse, the remarks after (\ref{def_ecoe}) imply that 
%$( \ers )_\mG \subseteq$ $( \rhors )_\eps$, and this proves (\ref{cd_2}). Finally, commutativity of the diagram (\ref{cd_1}) 
%follows by (\ref{def_sigmab}) and (\ref{def_sigma}).
%\end{proof}
%
%
%Note that in the proof of the previous theorem, we did not make use of invariant $\zro$-functionals on $\mG$. 
%
The previous theorem can be used as a starting point for the construction of an abstract duality theory for compact group bundles vs. tensor \sC categories with objects non-symmetric Hilbert bimodules. This shall be done in a forthcoming work.
%
%***************************************************************************
%

\end{document}